\documentclass[12pt,a4]{amsart}

\usepackage{amsmath}
\usepackage{amssymb}
\usepackage{amsthm}
\usepackage{enumerate}
\usepackage{url}
\usepackage{color}
\usepackage{dsfont}
\usepackage{epsfig}
\usepackage[hidelinks]{hyperref}
\usepackage{setspace}
\usepackage{comment}
\usepackage{booktabs}
\usepackage{float}
\usepackage{tikz}
\numberwithin{equation}{section}

\usepackage[font=small,format=plain,labelfont=bf,up,textfont=it,up]{caption}
\usepackage{subcaption}
\renewcommand{\S}{\mathcal{S}}

\usepackage{geometry}
\theoremstyle{plain}
\newtheorem{proposition}{Proposition}[section]

\newtheorem{theorem}{Theorem}[section]
\newtheorem{lemma}{Lemma}[section]
\newtheorem{corollary}{Corollary}[section]

\theoremstyle{definition}
\newtheorem{definition}{Definition}[section]

\theoremstyle{remark}
\newtheorem{rk}{Remark}[section]
\expandafter\let\expandafter\oldproof\csname\string\proof\endcsname

\usepackage{kbordermatrix}
\newcommand{\1}{\mathds{1}}
\newcommand{\mfo}{\mathfrak{o}}
\newcommand{\E}{\mathbb{E}}
\newcommand{\mP}{\mathbb{P}}
\renewcommand{\P}{\mathbb{P}}
\newcommand{\N}{\mathbb{N}}
\newcommand{\C}{\mathbb{C}}
\newcommand\norm[1]{\left\lVert#1\right\rVert}
\newcommand{\Em}{\overline{E}}

\newcommand{\be}{\begin{equation}}
\newcommand{\ee}{\end{equation}}
\newcommand{\by}{\begin{eqnarray*}}
\newcommand{\ey}{\end{eqnarray*}}

\renewcommand{\leq}{\leqslant}
\renewcommand{\geq}{\geqslant}
\usepackage{xcolor}
\definecolor{dark-red}{rgb}{0.4,0.15,0.15}
\definecolor{dark-blue}{rgb}{0.15,0.15,0.4}
\definecolor{medium-blue}{rgb}{0,0,0.5}
\hypersetup{
    colorlinks, linkcolor={dark-red},
    citecolor={dark-blue}, urlcolor={medium-blue}
}

\renewcommand{\l}{\mathfrak{l}}
\renewcommand{\r}{\mathfrak{r}}
\renewcommand{\Mc}{\mathcal{MC}}
\newcommand{\Hqy}{\mathbf{H}_q^{y]}}
\newcommand{\Hqb}{\mathbf{H}_q^{b]}}
\newcommand{\Sf}{\mathcal{M}}
\newcommand{\Bd}{\mathcal{B}}

\begin{document}
\title{Skip-free Markov chains}
\author{Michael C.H. Choi}
\address{Institute for Data and Decision Analytics, The Chinese University of Hong Kong, Shenzhen, Guangdong, 518172, P.R. China.}
\email{michaelchoi@cuhk.edu.cn}

\author{Pierre Patie}\thanks{The authors would like to thank an anonymous referee for his constructive comments and suggestions that improve substantially the quality of the paper.  The authors are also grateful to Jim Dai, Daniel Jerison, Laurent Miclo, Jim Pitman and Laurent Saloff-Coste for stimulating discussions on several aspects of this work. This work was partially supported by  NSF Grant DMS-1406599 and ARC IAPAS, a fund of the Communaut\'ee francaise de Belgique. The second author is grateful  for the hospitality of  the  LMA at the UPPA, where part of this work was completed.}

\address{School of Operations Research and Information Engineering, Cornell University, Ithaca, NY 14853.}
\email{pp396@cornell.edu}
\maketitle

\begin{abstract}
The aim of this paper is to develop a general theory for the class of skip-free Markov chains on denumerable state space. This encompasses their potential theory via an explicit characterization of their potential kernel expressed in terms of family of fundamental excessive functions, which are defined by means of the theory of Martin boundary.  We also describe their fluctuation theory generalizing the celebrated fluctuations identities that were obtained  by using the Wiener-Hopf factorization for the specific  skip-free random walks. 
We proceed by resorting to the concept of similarity to  identify the class of  skip-free Markov chains whose transition operator has only real and simple eigenvalues. We manage to find a set of sufficient and easy-to-check conditions on the one-step transition probability for a Markov chain to belong to this class. We also study several properties of this class including their spectral expansions given in terms of Riesz basis, derive a necessary and sufficient condition for this class to exhibit a separation cutoff, and give a tighter bound on its convergence rate to stationarity than existing results.
\end{abstract}


\section{Introduction}\label{sec:intro}
Let $X=(X_n)_{n\in \mathbb{N}}$ be a Markov chain on the countable state space $E = [[\l,\r]] \subseteq \mathbb{Z}$, where we use the notation $[[$ (resp.~$]]$) to denote that $\l$ (resp.~$\r$) may or may not be in $E$, defined on the filtered probability space $ (\Omega,(\mathcal{F}_n)_{n\in \mathbb{N}},\P=(\P)_{x\in \E})$.  We denote its transition matrix by $P = (p(x,y))_{x,y \in E}$. We assume further that $X$ is irreducible, i.e.~for all $x,y \in \dot{E}= ]\l,\r[$, $p^{(n)}(x,y)=\P_x(X_n=y)>0$ for some $n\in \N$, and $\textit{upward skip-free}$, i.e.~for all $x \in E$, $p(x,x+1)>0$ and $p(x,x+y)=0, y\geq 2$.
We denote by $\Sf$ be the set of such upward skip-free Markov chains (or transition operators) on $E$.

The aim of this paper to develop a comprehensive theory for the set $\Sf$, including {\emph{the potential theory, the fluctuation theory}} and, resorting to the algebraic  concept of similarity,  {\emph{the spectral theory}}. As a by-product, we also provide, for ergodic chains,  a detailed  analysis of the speed of convergence to stationarity and investigate the separation cutoff phenomena.

We recall that under the additional condition that $p(x,x-y)=0, y\geq 2$, that is, it is also skip-free to the left, $X \in \Sf$ becomes a birth-death chain. These chains have been and are still the object of intensive and fascinating studies. This probably originates from the seminal  work of Karlin and McGregor \cite{Karlin-McGregor_Spectral}, see also Lederman and Reuter \cite{LR54} and Anderson \cite{Anderson91}, on the diagonalization  of their transition operator  that provide deep insights into fine distributional properties of these chains.  Note that the spectral analysis of these operators has also revealed fascinating links with the theory of orthogonal polynomials and Stieljes moment problem. A review of birth-death chains including their potential theory,  is given below in Section \ref{subs:bd}.

For skip-free Markov chains however, the literature is much more limited than their birth-death counterparts even though skip-free chains appear naturally and frequently in theoretical and applied investigation. On the theoretical side, the (upward)  hitting time distribution of (upward) skip-free chains has
been studied by \cite{AW89, Fill}. Another notable research is the work of Adikahri \cite{Adikahri86} on excursion theory of such chains. In addition, Mao et al.~\cite{MZZ16} have recently investigated  the separation cutoff for skip-free chains. We also mention that Miclo and Patie \cite{Miclo-Patie18} initiated a new approach to study these chains based on intertwining relationships with some diffusions with jumps.  On the other hand, from an application viewpoint, skip-free chains have been widely used in population models, queueing theory and branching processes, see for example \cite{BGR82,ASM03,APW18} and the references therein.

Our contributions to the theory of the class $\Sf$ can be summarized as follows.
\begin{enumerate}[a)]
\setlength\itemsep{1em}

\item \textit{Potential theory}: We shall start our program by studying the potential theory of chains in $\Sf$. More specifically, we implement an original approach based on the theory of Martin boundary for Markov chains as developed by Dynkin \cite{dynkin} to express their $q$-potential kernel. In this vein, we recall that in the specific case of birth-death chains, this $q$-potential is given in terms of the two fundamental solutions (the decreasing and increasing one) of  a three-term recurrence equations (discrete analogue of a second order differential equation), see  Section \ref{subs:bd} for more details regarding this expression. In our context, the situation is more delicate as one has to solve  an infinite recurrence equation  whose set of solutions does  not seem to have been  clearly identified in the literature.  Although the issue of solving this equation is of algebraic nature, we shall elaborate a strategy based mostly on a combination of techniques from probability theory and  potential theory. More specifically, instead of trying to identify directly the  convex cone of $q$-excessive functions for the chain $X \in  \Sf$, we take an alternative route which consists in  characterizing   the $q$-potential kernel in terms of the so-called  fundamental $q$-excessive (for short F$q$E) functions  of three  different chains: $(X,\mP)$, $(X,\widehat{\mP})$ the dual chain as defined in \eqref{eq:dualdef} below and $(X,\mP^{y]}), y \in E$, where $(X,\mathbb{P}^{y]})$ is the Markov chain $(X,\mP)$  killed upon entering the  half-line $[[\l,y]$, which is plainly an upward skip-free Markov chain on the state space $E^{y]} = (y,\r]]$.

\item \textit{Fluctuation identities}: From this representation of the $q$-potential kernel, we derive the main fluctuations identities for chains in $\Sf$. We recall that the fluctuation theory, which is concerned  with the distribution of the first visit of the chain  to some (finite or infinite) intervals is of great importance in many applications such as biology, epidemiology and also in  ruin theory where  the upward skip-free property is required.  This theory is well established for the case of  birth-death chains  (skip-free on both sides), see e.g.~Karlin and McGregor \cite{Karlin-McGregor_Spectral}, Keilson \cite{Ke71}. On the other hand the famous Wiener-Hopf factorization technique has proven to be useful to characterize the law of the first exit times for the class of random walks, that is for Markov chains with stationary and independent increments, see Spitzer \cite{Spitzer} for further details on these identities. We also mention the interesting work of Fill \cite{Fill} where he characterizes the upward hitting time distribution of upward skip-free chain via establishing an intertwining with pure-birth chain. Our original approach goes decisively beyond these frameworks as it allows to treat in an unified way, not only the first upward passage time, but also the first downward passage time, including the possibility of undershoot, of the irreducible skip-free Markov chain. To implement our methodology on some specific examples, one merely needs to have access to a {\color{red}{transform}} (Laplace transform, Fourier transform, moment generating function\ldots) that determines the one-step transition kernel of the chain. We shall illustrate this idea by recovering, in a simple manner, Spitzer's identities for skip-free random walks, and, by studying the first passage times of branching Galton-Watson processes with immigration, whose details will be provided in a subsequent paper \cite{Patie-Wang}, see also \cite{APW18}. We also mention that the continuous analogue of these results  for skip-free continuous-time Markov processes on the real line has  been detailed in \cite{Patie-Vigon} and applications to generalized Ornstein-Uhlenbeck processes and continuous branching processes with immigration are carried out in \cite{Lefevre2013} and \cite{Patie-Wev} respectively.

\item \textit{Spectral theory and its applications}: We also aim at providing some insights into the spectral theory of the (transition operator of) Markov chains that  belong to the class $\Sf$. This is a  challenging issue as  the transition operator $P$ of such a chain is non-self-adjoint (non-reversible)  in the weighted Hilbert space $\ell^2(\pi),$  where $\pi$ is the reference measure, implying that there is no spectral theorem available for such bounded linear operator. To overcome this difficulty, we propose an original approach   based on the concept of similarity.  More specifically, we introduce a subclass of $\Sf$, denoted by $\S$,  whose each element is  related to a (diagonalizable) transition matrix of a birth-death Markov chain via a certain commutation relation. Using this identity, we resort to some techniques from non-harmonic analysis to investigate how to transfer the known spectral information of the reversible birth-death transition operator to the non-reversible one in order to obtain its spectral decomposition. By means of the inverse spectral theorem we manage to show that the class $\S$ characterizes completely the family of transition operators in $\Sf$ that have real and distinct eigenvalues. On the other hand, in Theorem \ref{thm:mcins}\eqref{it:Mc}, we shall provide  a set of sufficient and easy-to-check conditions on the one-step transition probability for a Markov chain to belong to $\S$.
It is worth mentioning that, to the best of our knowledge, this is the first identification of a class of non-reversible chains with such a spectrum. This is a useful fact which answers an open question raised by Fill \cite{Fill} on understanding the class of skip-free chains that have real and non-negative eigenvalues.

We believe that this new way of classifying Markov chains based on similarity orbit is powerful enough to tackle many substantial and delicate problems arising in the analysis of such chains. To illustrate this fact, we provide  the spectral decomposition of their non-symmetric transition matrices, the law of their first passage times, including the case with possible overshoot. We also study for ergodic chains in $\S$  the speed of convergence to equilibrium in both the $\ell^2(\pi)$-topology and the total variation distance. In this line of work, we indicate that Miclo \cite{Miclo15} has recently introduced a new notion known as Markov similarity and compares the mixing speed of Markov similar generators. As for the speed of convergence to equilibrium, there is a vast literature devoted to this important topic in various settings. For instance, in the case when $P$ is a linear self-adjoint operator in $\ell^2(\pi)$ (reversible), the chain satisfies the spectral gap, which according to  Roberts and Rosenthal \cite{RR}, is equivalent to $P$ being geometrically ergodic. In the non-reversible case, to overcome the lack of a spectral theory, many interesting reversibilization techniques have been implemented to obtain a quantitative rate of convergence and we refer the interested readers to Fill \cite{Fill91} and Saloff-Coste \cite{LSC97}. We propose an alternative approach  based on the concept of similarity which seems to be a natural extension of the spectral gap estimate developed in the reversible case. Moreover, our technique enables us to provide an explicit and a spectral interpretation as a perturbed spectral gap estimate of the (discrete) hypocoercivity phenomena that was introduced by Villani \cite{V09} for general non-self-adjoint semigroups.  We also manage to obtain generalization to the class $\mathcal{S}^M$, a subclass of $\mathcal{S}$ with stochastic monotone time-reversal to be formally introduced in Section \ref{sec:st}, of the remarkable \emph{spectral gap times mixing time going to infinity} separation cutoff criteria established by Diaconis and Saloff-Coste \cite{DSC06} for reversible birth-death chains, and recently by Mao et al. \cite{MZZ16} for continuous-time skip-free chain with stochastic monotone time-reversal.
\end{enumerate}

We point out that these three topics (potential, fluctuation and spectral theory)  are intertwined. 
 Indeed, there is  of course a  fundamental connection between the resolvent operator and the spectral theory of $P$ as the 
spectrum of $P$ is, by definition, the set of complex numbers $z$ such that the resolvent operator  $R_z=(zI - P)^{-1}$ does not exist or is unbounded. Note that since $P$ is a contraction operator, this set is included in the unit disc. Moreover, the development of the fluctuation identities is based on the expression of the $q$-potential kernels whereas the spectral theory allow us to get an explicit representation of the transition kernel (and hence by integration of the $q$-resolvent operator) and, also, of  the distribution of the  first passage times.

The rest of the paper is organized as follows. In Section \ref{sec:prelim}, we fix our notations and provide a  review of birth-death chains and Martin boundary theory. 
The definition and study of the fundamental $q$-excessive functions as well as the expression of the $q$-potential kernel are discussed in Section \ref{sec:pot_th}. In Section \ref{sec:fluc}, we  present the fluctuations identities, that is the explicit characterization (via the probability generating function) of the law of the first exit times of skip-free chains, and, following the line of work by Feller \cite{F66}, Kent and Longford \cite{KL83} and Viskov \cite{Viskov00}, a characterization of these stopping times as discrete infinitely divisible variables in Section \ref{sec:fpt}. We proceed in Sections \ref{sec:st}-\ref{sec:sc} by introducing the subclass of skip-free Markov chains that are similar to a diagonalizable birth-death chain and discuss their spectral properties as well as its applications to the speed of convergence to equilibrium and the study of the cutoff phenomena.

\section{Preliminaries}\label{sec:prelim}
In this Section, we review some classical concepts on Markov chains that will play a central role throughout the paper. This include some  facts of the potential theory and the Martin boundary theory of Markov chains.

\subsection{Basic facts on Markov chains}
We recall that $X$ is said to be \textit{upward skip-free} if the only upward transition is of unit size, yet it can have downward jump of any arbitrary magnitude, i.e.~for all $x\in E$ and $y \geq x+2$, $\P_x(X_{1}=y)=0$. We consider an upward skip-free irreducible Markov chain $X = (X_n)_{n \in \mathbb{N}_0}$ on a denumerable state space $E\subseteq \mathbb{Z}$ with left endpoint $\l$ and right endpoint $\r$. We use the convention that $\r \in E$ if the boundary point $\r$ is not absorbing. Otherwise, if $\r$ is absorbing or $\r = \infty$, we say that $X \in \Sf_{\infty}$. We assume through Section 2.1 that $X \in \Sf_{\infty}$  and postpone to Section \ref{sec:regular} the study of the Markov chain $X \notin \Sf_{\infty}$. 

Since $X$ is irreducible,  there exists $\pi$ a positive and finite-valued excessive measure for $P$, that is, $\pi P \leq \pi$, see \cite[Section $5.2$ and $6.8$]{KSK}, and that will serve as a reference measure.
Let $G_q$ be the $q$-potential kernel   of $X$ (or its Green function) with respect to the reference measure $\pi$. That is, for $0 < q < 1$,
\begin{eqnarray}\label{eq:defpk}
G_q(x,y) &=& \sum_{n = 0}^\infty q^n \frac{\mP_x(X_n = y)}{\pi(y)} = \sum_{n = 0}^\infty q^n \frac{p^{(n)}(x,y)}{\pi(y)}, \: x,y \in E.\end{eqnarray}
When $q = 1$, we write $G$, rather than $G_1$, to denote the 1-potential kernel. The $\pi$-dual matrix $\widehat{P} = (\widehat{p}(y,x))_{y,x \in E}$ is defined to be
	\begin{eqnarray}\label{eq:dualdef}
	\widehat{p}(y,x)\pi(y) &=&  p(x,y)\pi(x).\label{def:dual}
	\end{eqnarray}
Denote $\zeta$ to be the lifetime of $X$, that is, $\zeta(\omega) = k$ if $k$ is the last time that $X$ is in the state space $E$, and $\zeta(\omega) = \infty$ otherwise.

Let us now define the hitting times associated with $X$. Denote $T_A$ to be the first hitting time of the set A, that is, \[T_A = \inf\{n \geq 0; X_n \in A\},\] with the usual convention that $\inf \{\emptyset\} = \infty$. If $A = \{a\}$, we write $T_a = T_A$. Similarly, if $A = (\l,b]$, we use $T_{b]} = T_A$. Denote the first return time to be $T_a^+ = \inf\{n \geq 1; X_n = a\}$.

We also use, for   a real-valued function $f$ and a measure $\mu$ on $E$, the following notation, for any $x,y \in E$, 
\begin{align*}
	P f (x) &=  \sum_{y \in E} p(x,y) f(y),\\
 \mu P(y) &= \sum_{x \in E} \mu(x) p(x,y).
\end{align*}


\subsection{A review of  birth-death  chains} \label{subs:bd}
Let $Y$ be a birth-death chain on $E$ with  transition operator $Q \in \Bd$, that is,  $Q \in \Sf$ with the additional requirements that $Q(x,x-1)>0$ and $Q(x,x-y)=0$ for $y\geq 2$. $Q$ is a bounded self-adjoint  operator in the Hilbert space $\ell^2(\pi_Q)$ where $\pi_Q$ is the reversible measure, i.e.~for all $x,y\in E$, $Q(x,y)\pi_Q(x)=Q(y,x)\pi_Q(y)$.   The $q$-potential of $Q$, denoted by  $U_q$, is well-known to take the form, for any $x,y \in E$,
\begin{equation}\label{eq:gbd}
  U_q(x,y)= C_q F_q(x \wedge y) \widehat{F}_q(x \vee y),
\end{equation}
where $F_q$ (resp. $\widehat{F}_q$) is the unique increasing (resp.~decreasing) solution to the equation $q Q F = F$ satisfying  appropriate boundary conditions, and $C_q$ is the inverse of their Wronskrian. Note that this equation boils down to a three-term recurrence equations, see  \cite[Ex.~5.3 p.150 and Section 5.4]{DM76} for details regarding this expression. This is reminiscent of the expression of the potential kernel of the  continuous time-space  analogue, whose generator is a second order differential operator. We recall that a systematic and thorough study of one-dimensional diffusion has been undertaken by Feller \cite{Fe54}.
Moreover, the moment generating function of the first hitting time $T^Y_a=\inf\{ n\geq 0;\: Y_n=a\}$  of $Y$   to a fixed level $a \in E$ is given by
\begin{equation}\label{eq:gbd}
  \mathbb{E}_x[q^{T^Y_a}]= \frac{F_q(x)}{F_q(a)}\1_{\{x\leq a\}} +\frac{\widehat{F}_q(x)}{\widehat{F}_q(a)}\1_{\{x> a\}}.
\end{equation}
Let us now describe a link between this expression and the diagonalization of the  operator $Q$ in $\Bd$.  We assume, for sake of simplicity, that $\mathfrak{l}=0$ and thus $E$ is countable subset of $\mathbb{N}$. It is worth mentioning that the similarity transform $D_{\sqrt{\pi}} Q D_{\frac{1}{\sqrt{\pi}}}$, where $D_a$ is the diagonal matrix of $a$ yields to a symmetric tridiagonal matrix which belongs to  the well-studied class of Jacobi matrices. We also recall that  a function $\phi$ is  a Pick function  if $\phi$ admits an analytical continuation to the cut complex plane $\mathbb{C}\setminus [0,\infty)$ such that $\Im \phi(z)\Im(z)\geq 0$, and we denote by  $\mathcal{P}$ the subset of such Pick functions  that admit the representation, for $\Im(z)>0$,
 \begin{equation} \label{eq:dp}
  \phi(z) = \int_{-1}^{1}\frac{d\Delta(r)}{r-z},
 \end{equation}
 where $\Delta$ is a probability measure on $[-1,1]$.   From the work of Karlin and McGregor \cite{Karlin-McGregor_Spectral}, we know that there exists a spectral mapping $\mathrm{K}: \Bd \rightarrow \mathcal{P}$ which is one-to-one and $\Delta$   is the spectral measure of $Q$. More specifically, we have for any $f \in \ell^2(\pi_Q)$ and $n\in \mathbb{N}$, $Q^n$ can be diagonalizable as follows
   \begin{equation}\label{eq:sbd}
     Q^n f= \int_{-1}^{1} r^n \langle f , \mathrm{F}_r \rangle_{\pi_Q} \mathrm{F} _r \: d\Delta(r),
   \end{equation}
   where for  $r \in supp(\Delta)$,  $Q\mathrm{F}_r=r\mathrm{F}_r$. Note that  $S(Q)$ the spectrum of $Q$ is such that $S(Q)=supp(\Delta)\subset [-||Q||,||Q||] \subseteq [-1,1]$.
    Another  remarkable and deep result, which was conjectured by Krein, is the onto property of the spectral map $\mathrm{K}$, see \cite[Appendix 3]{Gorbachuk1997} for an historical account. Indeed for any $\phi \in \mathcal{P}$, i.e.~a  Pick function of the form  \eqref{eq:dp}, there exists a unique $Q \in \Bd$  on $E$ with  $\mathfrak{l}$ a non-killing boundary   with a spectral representation of the form \eqref{eq:sbd}. The first proof
of this inverse spectral result was given in the monograph by  H.~Dym
and H.P.~McKean \cite[Chapter 6]{DM76}. It uses the theory of Hilbert spaces of entire
functions, and a deep uniqueness theorem due to de Branges. Note also that such a spectral representation reveals that when the spectrum of $Q$ is composed of isolated eigenvalues then they are necessarily simple (see e.g.~\cite[Section $5.8$]{DM76}).
 \subsection{Martin boundary theory of denumerable Markov chains}

In this subsection, we review some essential results of Martin boundary theory of denumerable irreducible Markov chains based on \cite{dynkin,KSK,Woess}. First, we recall the definition of $q$-excessive,  $q$-harmonic and $q$-potential functions. A non-negative function $f$ on $E$ is $q$-excessive (resp.~$q$-harmonic) if $q P f (x) \leq  f(x)\,(\mathrm{resp.}~q P f (x) = f(x))$ for all $x \in E$, where $0 < q \leq 1$. A non-negative function $f$ on $E$ is a $q$-potential if $f$ is $q$-excessive and $\lim_{n \to \infty} (qP)^n f(x) = 0$ for all $x \in E$. We say that $f$ is harmonic (resp.~excessive) if $f$ is $1$-harmonic (resp.~$1$-excessive).  We write $$\mathcal{E}_q = \{f:E\to\mathbb{R}^+;~qPf \leq f\} \, \mathrm{(resp.}~\mathcal{H}_q, \mathcal{P}_q)$$ for the set of $q$-excessive (resp.~$q$-harmonic, $q$-potential) functions on $E$. We simply write $\mathcal{E}$ (resp.~$\mathcal{H}$, $\mathcal{P}$) to denote the set of excessive (resp.~harmonic, potential) functions.  Note that the irreducibility property implies that $f \in \mathcal{E}_q$ is positive, since if there exists $y \in E$ such that $f(y) > 0$, then by irreducibility there exists $n$ such that for $x \in E$ $f(x) \geq q p^n(x,y) f(y) > 0$. We will also use the notation $\widehat{\mathcal{E}}_q$ (resp.~$\widehat{\mathcal{H}}_q, \widehat{\mathcal{P}}_q$) to denote the set of $q$-excessive (resp.~$q$-harmonic, $q$-potential) functions associated with $\widehat{P}$, which was defined in \eqref{eq:dualdef}. We point out  that if $\hat{h} \in \widehat{\mathcal{E}}_q$, then $\hat{h} \pi$ is a $q$-excessive measure for $P$ in the sense that $q \: \hat{h} \pi P f\leq \hat{h} \pi  f$.
%
	Another commonly used terminology for excessive (resp.~ harmonic, potential) function is superharmonic (resp.~ invariant, purely-excessive) function.

We further recall the definition of minimal $q$-harmonic function. A non-zero function $f$ on $E$ is minimal $q$-excessive if $f = f_1 + f_2$ implies $f = c_i f_i$ for $i = 1,2$, where $f_1, f_2 \in \mathcal{E}_q$, $c_1$ and $c_2$ are constants and $0 < q \leq 1$. We say that $f$ is minimal excessive if $f$ is minimal $1$-excessive. We write $\mathcal{E}_q^{min}$ (resp.~$\mathcal{H}_q^{min}$) to be the set of minimal $q$-excessive (resp.~minimal $q$-harmonic) functions on $E$. For further discussion on minimal excessive and minimal harmonic functions, we refer the interested readers to \cite[Chapter $7$]{Woess} and \cite[Section $14$]{dynkin}. Next, we state the classical Riesz representation theorem (see e.g. \cite[Theorem $6$]{dynkin}), which gives an unique decomposition of excessive function.
\begin{theorem}[Riesz representation theorem]\label{thm:riesz}
	For $0 < q \leq 1$, every $q$-excessive function can be written uniquely as the sum of a $q$-potential and a $q$-harmonic function. That is, if $f \in \mathcal{E}_q$, then
	$$f(x) = \sum_{y \in E} G_q(x,y) k_q(y) \pi(y) + h_q(x),$$
	where $k_q(x) = f(x) - qPf(x)$ and $h_q(x) = \lim_{n \to \infty} (qP)^n f(x) \in \mathcal{H}_q$.
\end{theorem}
Recall that $X \in \mathcal{M}_{\infty}$  with transition matrix $P$ and reference measure $\pi$ on a denumerable state space $E$. For a measure $\mu$ on $E$, define $E_{\mu} = \{y \in E;~ \mu G_q(y) > 0\}$, where $\mu G_q(y) = \sum_{x \in E} \mu(x) G_q(x,y)$, for a $0 < q <1$. We say that the measure $\mu$ is  a \textit{standard measure} if $E_{\mu} = E$. For $x,y \in E$ and $0 < q <1$ (with also $q=1$ when $(X,\mP)$ is a transient), the $q$-Martin kernel associated to a standard measure $\mu$ is defined to be
\begin{align} \label{eq:defmk}
{}_\mu K_{q}(x,y) &= \dfrac{G_q(x,y)}{\mu G_q(y)},
\end{align}
where the denominator is positive since $\mu$ is standard. Note that ${}_\mu K_{q}(x,y)$ is finite for any $x,y \in E$. Next, define a metric ${}_\mu d_{q}$ on the space $E$ by
\begin{align*}
{}_\mu d_{q}(y,z) = \sum_{x \in E} w_x \left(|{}_\mu K_{q}(x,y) - {}_\mu K_{q}(x,z)| + |\1_{\{x = y\}} - \1_{\{x = z\}}| \right),
\end{align*}
where the weights $(w_x)_{x \in E}$, with $w_x > 0$, are chosen such that $\sum_{x \in E} w_x < \infty$. 
We can obtain the completion of $E$ with respect to the metric ${}_\mu d_{q}$, namely $\Em$, and the boundary of $E$ in $\Em$ is denoted as $\partial E = \Em - E$. $\Em$ is the Martin compactification of $E$ and $\partial E$ is the Martin boundary. The set $$\partial_{\P} E = \{y \in \partial E; \: {}_{\mu}K_q(x,y) \, \text{is minimal } q \text{-harmonic in}\, x\}$$ is known as the \textit{minimal Martin boundary}. When there is no ambiguity in the probability measure, we write $\partial_m E = \partial_{\P} E$. The inclusion of the indicator terms $\1_{\{x = y\}}$ and $\1_{\{x = z\}}$ in the metric ${}_\mu d_{q}$ ensures that $E$ is an open set in the Martin compactification $\Em$, and the Martin boundary $\partial E$ is closed. Next, fix a reference point $\mathfrak{o} < \r - 1$, and we write the $q$-Martin kernel $_{\delta_{\mathfrak{o}}}K_q(x,y)$ as
\begin{align*}
	_{\delta_{\mathfrak{o}}} K_q(x,y) &= \dfrac{\mu G_q (y)}{G_q(\mathfrak{o},y)} {}_{\mu} K_{q}(x,y) = \dfrac{_{\mu} K_{q}(x,y)}{_{\mu} K_q(\mathfrak{o},y)}.
\end{align*}
Similarly, we have
\begin{align*}
	_{\mu} K_q(x,y) &= \dfrac{G_q(\mathfrak{o},y)}{\mu G_q (y)} {}_{\delta_{\mathfrak{o}}} K_q(x,y) = \dfrac{{}_{\delta_{\mathfrak{o}}} K_q(x,y)}{{}_\mu K_q(y)}.
\end{align*}
A sequence $(y_n)$ is a Cauchy sequence in the metric space $(E,{}_\mu d_{q})$ if and only if $(_{\mu} K_{q}(x,y_n))$ is a Cauchy sequence of real numbers for every $x$ if and only if $(_{\delta_{\mathfrak{o}}} K_q(x,y_n))$ is a Cauchy sequence of real numbers for every $x$ if and only if $(y_n)$ is a Cauchy sequence in the metric space $(E,{}_{\delta_{\mathfrak{o}}}d_q)$. Thus, up to homeomorphism, $\Em$ is independent of the choice of  the initial measure and of the reference point $\mathfrak{o}$. It can also be shown that $\Em$ is independent of the choice of the weights $(w_x)_{x \in E}$ (see e.g. Proposition 10.13 in \cite{KSK}). From now on, we fix the reference point $\mathfrak{o}$ and write for all $x,y \in E$,
\begin{align}\label{def:Kqxy}
K_q(x,y) &= {}_{\delta_{\mfo}} K_q(x,y) = \dfrac{G_q(x,y)}{G_q(\mfo,y)}.
\end{align}
Let $\Omega_{\infty} = \{\omega; \:  \text{there is}\, x_\infty \in \partial E \, \text{such that} \, x_n \rightarrow x_\infty \, \text{in the Martin topology} \}.$
$\Omega_{\infty}$ can be interpreted as the set of \textit{non-terminating} trajectories of $X$ that converges to $\partial E$.
If $(X,P)$ is transient and $P$ is a stochastic matrix then there is a random variable $X_{\infty}$ taking values in $\partial E$ such that for each $x \in E$,
$\mP_x \left(\lim_{n \rightarrow \infty} X_n = X_{\infty} \right) = 1.$
In terms of trajectory space, we have
$\mP_x \left(\Omega_{\infty}\right) = 1.$
If $P$ is strictly substochastic at some states, we should extend the trajectories to $E \cup \{\nabla\}$, where $\nabla$ is the graveyard state. Define
$$\Omega_{\nabla} = \{\omega \in \Omega;\: \, \text{there is} \, k \geq 1 \, \text{such that} \, x_n \in E, \, \forall n \leq k \, \text{and} \, x_n = \nabla, \forall n \geq k + 1  \}.$$
$\Omega_{\nabla}$ is the set of trajectories that eventually reach $\nabla$. Denote $\zeta$ to be the lifetime of $X$, that is, $\zeta(\omega) = \mathfrak{K}$ for $\omega \in \Omega_{\nabla}$, where $\mathfrak{K}$ is the last time that $X$ is in the state space $E$ (as defined in $\Omega_{\nabla}$), and $\zeta(\omega) = \infty$ otherwise. Define
$\Omega_{\zeta} = \Omega_{\nabla} \cup \Omega_{\infty}$.
If $(X,\mP)$ is transient then there is a random variable $X_{\zeta}$ taking values in $\Em$ such that for each $x \in E$,
$\mP_x \left(\lim_{n \rightarrow \zeta} X_n = X_{\zeta} \right) = 1.$
In terms of trajectory space, this means that
$\mP_x \left(\Omega_{\zeta}\right) = 1.$

Next, suppose now that $h \in \mathcal{E}_q$. The Doob transform or $h$-process of $X$ is defined to be the Markov chain on $E^h = \{y \in E;~ h(y) > 0 \}=  E$, by irreducibility, with the canonical measure $\mP^h_x$ such that
$\mP^h_x(X_1 = y) = \dfrac{p(x,y)h(y)}{h(x)}.$
Recalling that $\mfo$ is the fixed reference point, the $q$-potential and $q$-Martin kernels associated with the $h$-process take respectively the form
\begin{align}
	G^h_q(x,y) &= \dfrac{G_q(x,y)h(y)}{h(x)}, \\
	K^h_q(x,y) &= \dfrac{K_q(x,y)h(\mfo)}{h(x)}. \label{eq:Kh}
\end{align}
From \eqref{eq:Kh} and the definition of the metric $_{\mu} d_{q}$, we observe that the Martin compactification $\Em$ is homeomorphic to the Martin compactification of the $h$-process.
 Next, we state the following which are the main claims of Theorem $6$ and $7$ in \cite{dynkin}.
\begin{theorem}[Uniqueness of the representation]\label{thm:uniq} \label{thm:htransp}
Let $q \in (0,1]$. If $h \in \mathcal{E}_q$ such that $h(\mfo)=1$ then $h$ has a unique representation of the form
$$h(x) = \int_{E \cup \partial_m E} K_q(x,y) \, \mu_h(dy) = K_q\mu_h(x),$$
where $\mu_h(\cdot) =  \mP_1^h(X_{\zeta} \in \cdot)$ defines  a probability measure.
Conversely, for any finite measure $\mu$, the mapping $x \mapsto \int_{E \cup \partial_m E} K_q(x,y) \, d \mu(y)$ defines an $q$-excessive function, which is $q$-harmonic if and only if $\mu(E) = 0$.
Finally, for all $y \in E \cup \partial_m E$, let $h^y(\cdot) = K(\cdot,y)$. Then $y \in E \cup \partial_m E$ if and only if $\mP^{h^y}(X_{\zeta} = y) = 1$. Moreover we have $\mP^{h^y}(\zeta = \infty) = 1$ if and only if $y \in \partial_m E$.
\end{theorem}
The previous claim means that $X$ is forced to terminate at the point $y \in E \cup \partial_m E$ $\,$ $\mP_x^{h^y}$-almost surely.

\begin{theorem}\label{thm:minE}
We have $\mathcal{E}_q^{min}=\{h_q; \: h_q(x) = C K_q(x,y), C>0 \textrm{ and } y \in E \cup \partial_m E\}$.
\end{theorem}
Finally, we recall the following useful result whose proof follows readily from \cite[Theorem $11.9$]{CW05}.
\begin{lemma}\label{lem:CW}
	Suppose that $h_q \in \mathcal{E}_q$, and $T$ is a stopping time with respect to $(\mathcal{F}_n)_{n \geq 1}$, the sigma field generated by $(X_n)$. For $x \in E^{h_q}$,
		$$\mP_x^{h_q} (T < \infty) = \dfrac{1}{h_q(x)}\E_x \left[ q^T h_q(X_T) \1_{\{T < \infty\}}\right].$$
\end{lemma}

\section{Potential Theory} \label{sec:pot_th}

In this Section, we  provide an explicit representation of the $q$-potential kernel $G_q$, as defined in \eqref{eq:defpk},  of an upward skip-free Markov chain $X \in  \Sf$. We recall that in the simpler case when $X$ is a birth-death Markov chain, i.e.~skip-free in both directions, then its $q$-potential kernel takes the form
\begin{equation} \label{eq:defpbd}
  U_q(x,y)= C_q F_q(x \wedge y) \widehat{F}_q(x \vee y),
\end{equation}
where $F_q$ and $\widehat{F}_q$ are the two fundamental solutions of  a three-term recurrence equations (discrete analogue of a second order differential equation), see  Section \ref{subs:bd} for more details regarding this expression. In our context, the situation is more delicate as one has to solve  an infinite recurrence equation  whose set of solutions does  not seem to have been  clearly identified in the literature.  Although the issue of solving this equation is of algebraic nature, we shall elaborate a strategy based mostly on a combination of techniques from probability theory and  potential theory. We start by  
expressing the $q$-potential kernel in terms of the so-called  fundamental $q$-excessive (for short F$q$E) functions  of the following three processes: $(X,\mP)$, $(X,\widehat{\mP})$ and $(X,\mP^{y]})$, where $(X,\mathbb{P}^{y]})$ is the Markov chain $(X,\mP)$  killed upon entering the  half-line $[[\l,y]$, which is plainly an upward skip-free Markov chain on the state space $E^{y]} = (y,\r]]$,  with transition kernel denoted by  $P^{y]}$. We are now ready to state the main result of this Section.


\begin{theorem}\label{thm:green2}
	Suppose that $X \in \mathcal{M}_{\infty}$.
\begin{enumerate}
\item \label{it:mb}
 \label{it:ph} Writing, for any $x\in E$ and $0<q<1$, \begin{equation} H_q(x) =  K_q\delta_\r(x), \end{equation}
  (resp.~$\widehat{H}_q(x) =
    \widehat{K}_q\delta_\l(x)$) with $\delta_{\r}$ is the Dirac mass at $\r$ and $K_q$ defined in \eqref{def:Kqxy}, we have, for all $0 < q < 1$,  that \[H_q \in \mathcal{E}_q^{min}\] (resp.~$\widehat{H}_q \in \widehat{\mathcal{E}}_q^{min}$)  and it is the unique minimal increasing (resp.~decreasing) $q$-excessive for $P$ (resp.~$\widehat{P}$)  such that  $H_q(\mfo) = 1$ (resp.~$\widehat{H}_q(\mfo)=1$). Moreover,  if $X \in \Sf_{\infty}$ (resp.~$X \in \Sf \setminus \Sf_{\infty}$) then $H_q $ is the unique increasing function in $\mathcal{H}_q$ (resp.~$H_q \in \mathcal{P}_q$  with $H_q(\r)<\infty$).
   \item  \label{it:phk}
    For any $y<\r$, $0<\kappa_q^{y]}=\lim\limits_{x\to \r} \frac{K_q\delta_\r(x)}{K^{y]}_q\delta_\r(x)} <\infty $. Then the function  $\Hqy$ defined, for any $x \in E^{y]}$, by
    \begin{equation} \Hqy(x) =  \kappa_q^{y]} K^{y]}_q\delta_\r(x) \end{equation}  has  (with respect to $P^{y]}$) the same property as $H_q$ but with the normalization    \[\lim\limits_{x\to \r} \frac{\Hqy(x)}{H_q(x)}=1\]
    (recall that by convention $\Hqy(x) =0$ for any $x \leq y$).
		\item \label{it:poh}Finally, 
set  $C_q = G_q(\mfo,\mfo)\,>0$. Then, we have, for all $x,y \in E,$
		\begin{equation}\label{eq:bargxysum}
		G_q(x,y)=
		C_q \: \widehat{H}_q(y)\left(H_q(x) -  \Hqy(x)\right).
		\end{equation}
	\end{enumerate}
\end{theorem}


\begin{rk}
Up to minor modifications all statements hold for $q=1$ when $(X,\mP)$ is transient.	
\end{rk}

\begin{rk}
The terminology F$q$E comes from the birth-death case where these functions are usually called the fundamental solutions of the associated three-term recurrence equation and are $q$-excessive.	
\end{rk}
\begin{rk}

Recall, from \eqref{eq:gbd}, that  the $q$-potential of a  birth-death chain is given in terms of two excessive functions. Each of them is associated to one point in the Martin boundary which is, due to its skip-free property in both directions, reduced to two points. In our situation, due to the random size of downwards jumps, the description of the Martin boundary below is more delicate. To overcome this difficulty, we start by taking advantage of the upward skip-free property to identify, as in the case of birth-death chain, the $q$-excessive function $H_q$ attached to the unique point above of the Martin boundary. Then, we use the dual chain  which is downward skip-free to identify, in the dual way, the $q$-excessive function $\widehat{H}_q$ associated to the  unique point below of the Martin boundary.  Finally, instead of trying to describe the Martin boundary below for the original chain,  we introduce the Markov chain killed at the first passage time at a point, say $y$, below the starting  point. This latter being also an upward skip-free Markov chain,  we identify, in a similar way than for the original chain, its $q$-excessive function $\Hqy$ associated to the point above of the Martin boundary.  It turns out that these three $q$-excessive functions coming from three different chains  enable us to characterize the $q$-potential kernel at a fixed point for the original chain.
\end{rk}
\begin{rk}\label{rk:pk}
The expression \eqref{eq:bargxysum} of the $q$-potential  is expressed in terms of the function $H_q$ which itself is defined in terms of the Martin kernel, which is a normalized version of the $q$-potential. Although this may sound a bit awkward, this representation is very useful when applied to specific (and solvable) instances.  Indeed, in such case, the functions $H_q$  can  be computed by either solving the equation $q P f =  f$ which admits an unique increasing solution when $H_q  \in \mathcal{H}_q$, see Section \ref{sec:sfrw} for an application of this idea to the class of skip-free random walks. Otherwise, in the other case, i.e.~$H_q  \notin \mathcal{H}_q$ but is purely $q$-excessive, its computation can be realized through some asymptotic analysis carried out on some transforms  determining the transition kernel, such as the moment generating functions. This latter approach has been exploited in  \cite{APW18} to compute the corresponding function $H_q$ for  (continuous-time) branching process, see also \cite{Patie-Wang} for the study  of the potential and fluctuation theory of these chains including immigration. We also refer to Remark \ref{rk:ft} for further discussion on how to use  our approach to obtain the fluctuation identities of skip-free Markov chains.
\end{rk}
\begin{rk}
	Additional properties of $H_q$ and its relation with infinite divisibility are studied in Section \ref{sec:fpt} Corollary \ref{it:id}.
\end{rk}



We proceed with the proof of these statements which is split into several intermediate results.
\subsection{Proof of Theorem \ref{thm:green2}}
We start with the following result which relates the Martin kernel to the hitting time distribution.

\begin{lemma}\label{thm:hit}
		For any $x,y \in E$,
		\[	\E_x(q^{T_y}) = \dfrac{K_q\delta_y(x)}{K_q\delta_y(y)}. \]
\end{lemma}

\begin{proof}
 By Theorem \ref{thm:riesz}, for any $y\in E$, $K_q\delta_y \in \mathcal{P}_q$ which leads, by Theorem \ref{thm:htransp},  to
	$\mP^{K_q\delta_y}(X_{\zeta} = y, \zeta < \infty) = 1$, that is, for any $x,y\in E$,
	\[ \mP_x^{K_q\delta_y}(T_y < +\infty) = 1. \]
Since, on the other hand, by  Lemma \ref{lem:CW}, we have, for any $x,y\in E$,
\[ \mP_x^{K_q\delta_y}(T_y < +\infty)= \E_x(q^{T_y})\dfrac{K_q\delta_y(y)}{K_q\delta_y(x)} \]
we complete the proof.
\end{proof}

\subsubsection{Proof of Theorem \ref{thm:green2}\eqref{it:mb}} Suppose that $x \vee \mfo \leq y \leq a $, where $x \vee \mfo = \max \{x,\mfo\}$. By means of Lemma \ref{thm:hit}, the upward skip-free property and the strong Markov property, we obtain that
	\begin{align}\label{eq:ratiogenTaK}
	K_q \delta_a(x)=\frac{\mathbb{E}_x(q^{T_a})}{\mathbb{E}_\mfo(q^{T_a})} &=\frac{\mathbb{E}_x(q^{T_y})\mathbb{E}_y(q^{T_a})}{\mathbb{E}_\mfo(q^{T_y})\mathbb{E}_y(q^{T_a})}=
	K_q \delta_y(x).
	\end{align}
Thus, for any $y\geq x \vee \mfo,\: K_q \delta_y(x)=K_q (x, x \vee \mfo)$. Hence, one can trivially define the function $H_q$ as the extended Martin kernel at $\r$, that is, for $x \in E$,
	$$H_q(x) = \lim_{y \to \r} K_q(x,y)= \int_{E \cup \partial_m E} K_q(x,y) \delta_{\r}(dy).$$
Hence if $X \in \Sf_{\infty}$ (resp.~$X \in \Sf \setminus \Sf_{\infty}$) then $\r \in \partial_{\P} E $ (resp.~$\r \in E$) and thus by theorems \ref{thm:uniq} and  \ref{thm:minE} (resp.~ and Theorem \ref{thm:riesz}), we get that   $H_q \in \mathcal{H}_q \cap \mathcal{E}_q^{min}$  (resp.~$H_q \in \mathcal{P}_q \cap \mathcal{E}_q^{min}$).
	Next note, from the first identity in \eqref{eq:ratiogenTaK}, that $H_q(\mfo) = 1$, and, for $x \leq y$,  $H_q(x)=K_q \delta_y(x)$. Hence, by Lemma \ref{thm:hit} we have, for any $x \leq y$,
	\begin{equation}\label{eq:uphit}
		\E_x(q^{T_y}) = \dfrac{H_q(x)}{H_q(y)},
	\end{equation}
which entails, by the irreducibility of $X$, that $H_q$ is positive everywhere since for any $x \in E$, the ratio $H_q(x)=\frac{\mathbb{E}_x(q^{T_a})}{\mathbb{E}_\mfo(q^{T_a})} > 0$. To see that the mapping $x \mapsto H_q(x)$ is increasing, one observes from again the strong Markov property and the upward skip-free property that (recall that $x  \leq y \leq a$)
	$$H_q(x) = \frac{\E_x(q^{T_y}) \E_y(q^{T_a})}{\E_\mfo(q^{T_a})} = \E_x(q^{T_y}) H_q(y) <  H_q(y).$$
To prove the uniqueness, we proceed by contradiction and thus assume that there exists a positive function $h_q \in \mathcal{E}_q^{min}$ (resp.~in $\mathcal{H}_q$ when $X\in \Sf_{\infty}$) which differs from  $H_q$ and  which is also an  increasing function on $E$. Then, according to Theorem \ref{thm:minE},  there exists $y_0 \in E$ (resp.~$y_0=\l$ or $y_0=\r$) such that for all $x\in E$, $h_q(x)=\frac{K_q(x,y_0)}{K_q(o,y_0)}$. Thus, on the one hand, from Lemma \ref{thm:hit} we deduce that for any $x\leq y_0$,
\[	\E_x(q^{T_y}) = \dfrac{K_q\delta_{y_0}(x)}{K_q\delta_{y_0}(y_0)}=\frac{h_q(x)}{h_q(y_0)}. \]
 This combines with \eqref{eq:uphit} yield  that $h_q(x)=H_q(x)$  for any $x\leq y_0$, which proves the claim when $H_q\in \mathcal{H}_q$ and $y_0=\r$. In the other cases, choose $x> y_0$ such that $h_q(x)\neq H_q(x)$. Then, observe from Theorem   \ref{thm:uniq} that $ \mP_x^{h_q}(T_{y_0} < +\infty) =1$. As   $ \mP_x^{h_q}(T_{y_0} < +\infty)=\frac{h_q(y_0)}{h_q(x)}\E_x(q^{T_{y_0}})$ 
  and $h_q$ is increasing and $x>y_0$, we get that $\E_x(q^{T_{y_0}})>1$ which is impossible. This completes the uniqueness property of $H_q$. 
To complete the proof of Theorem \ref{thm:green2}\eqref{it:mb}, we use similar arguments for deriving the stated properties of $\widehat{H}_q$ after recalling that the  dual chain $(X,\widehat{\mP})$ is downward skip-free.

\subsubsection{Proof of Theorem \ref{thm:green2}\eqref{it:phk}}
We start with the following claim which is a straightforward reformulation of  Theorem \ref{thm:green2}\eqref{it:mb} for the killed chains.

\begin{lemma}\label{lem:minEb}
Let $b \in E$ and choose a reference point $\mfo^{b]} \in E^{b]}$. Then, for all $0 < q <1$, the  function $H_q^{b]}(x) =  K^{b]}_q\delta_\r(x)$ defined on $E^{b]}$  is positive on $E^{b]}$, minimal, increasing $q$-harmonic for $P^{b]}$ with $H_q^{b]}(\mfo^{b]}) = 1$. Moreover, for any $b\leq x \leq a$,
	$$\mathbb{E}_x^{b]}(q^{T_a}) = \mathbb{E}_x(q^{T_a} \1_{\{T_a < T_{b]}\}}) = \dfrac{H_q^{b]}(x)}{H_q^{b]}(a)}.$$
\end{lemma}
\begin{proof}
Under $\mP^{b]}$, $X$ is an upward skip-free Markov chain on $E^{b]}$, the results follows from Theorem \ref{thm:green2}\eqref{it:mb} and the identity \eqref{eq:uphit}.
\end{proof}

The following lemma expresses the pgf of the \emph{downward} hitting times $(T_b, \mP^{[a}_x)$, where $a> x > b$, in terms of the F$q$E functions of $(X,\mP)$ and $(X,\mP^{b]})$.

\begin{lemma}\label{lem: below}
For any  $b < x < a$ and $0<q<1$, we have
	$$\mathbb{E}_x(q^{T_b} \1_{\{T_b < T_a\}}) = \dfrac{H_q(x)}{H_q(b)} - \dfrac{H_q(a)}{H_q(b)} \dfrac{H_q^{b]}(x)}{H_q^{b]}(a)}.$$
\end{lemma}
\begin{proof}
Since  by definition $H_q =  K_q\delta_\r$, we have, from Theorem \ref{thm:htransp} that,
\begin{equation} \label{eq:h1}
\mP^{H_q}(X_{\zeta} = \r) = 1.
\end{equation}
This combines with the upward skip-free property, see Lemma \ref{lem:CW}, yields that under $\mathbb{P}_x^{H_q}$ the sample paths of $X$ that drop below $b$ before hitting $a$ must reaches $b$ before reaching $a$.  That is \begin{equation}\label{eq:trick}
  \mathbb{P}_x^{H_q}(T_b < T_a)=\mathbb{P}_x^{H_q}(T_{b]} < T_a) = 1 - \mathbb{P}_x^{H_q}(T_a < T_{b]}),
\end{equation}  where we used again \eqref{eq:h1} for the second identity. Hence, an application of Lemma \ref{lem:CW} gives
	\begin{align*}
	\mathbb{P}_x^{H_q}(T_b < T_a)=\dfrac{H_q(b)}{H_q(x)} \mathbb{E}_x(q^{T_b} \1_{\{T_b < T_a\}}) &= 1 - \dfrac{H_q(a)}{H_q(x)} \mathbb{E}_x(q^{T_a}\1_{\{T_a < T_{b]}\}}) = 1 - \dfrac{H_q(a)}{H_q(x)}\dfrac{H_q^{b]}(x)}{H_q^{b]}(a)},
	\end{align*}
	where the last equality follows from Lemma \ref{lem:minEb}. Rearranging the terms provides the desired result.
\end{proof}
The proof of Theorem \ref{thm:green2}\eqref{it:phk} follows readily  after the following claim.
\begin{lemma}\label{lem:Kqb}
	 For $x \in E^{b]}$, define
	$$\kappa_q^{b]}(x) = \dfrac{H_q(x)}{H_q^{b]}(x)}.$$
	Then the mapping $x \mapsto \kappa_q^{b]}(x)$ is non-increasing  on $E^{b]}$ with $0 < \kappa_q^{b]}(x) < \infty$. Furthermore, $0<\kappa_q^{b]} = \lim_{x \to \r} \kappa_q^{b]}(x) < \infty.$
\end{lemma}
\begin{proof}
	It is clear that, for all  $x \in E^{b]}$, $0 < \kappa_q^{b]}(x) < \infty$, since both $H_q$ and $H_q^{b]}$ are positive and finite on $E^{b]}$.
	Next, for any $x \in E^{b]}$,
	$$qP^{b]}H_q(x) \leq qPH_q(x) \leq H_q(x),$$
	where the first inequality follows from the fact that $P^{b]}$ is the restriction of $P$ to $E^{b]}$, and we use that $H_q \in \mathcal{E}_q$ in the second inequality. Therefore, $H_q$ (restricted on $E^{b]}$) is $q$-excessive for $P^{b]}$. Thus, one may define the $H_q$-transform of $qP^{b]}$ by
	$$^{H_q}P^{b]}(x,y) = \dfrac{H_q(y)}{H_q(x)} q P^{b]}(x,y),$$
	where $x,y \in \{x \in E^{b]}:\: H_q(x) > 0\} = E^{b]}$ by Theorem \ref{thm:hit}. Using Lemma \ref{lem:CW} and Lemma \ref{lem:minEb},  we have, for any $x \leq a$,
	\begin{align}\label{eq:HqPb}
	^{H_q}\mP_x^{b]}(T_a < \infty) &= \dfrac{H_q(a)}{H_q(x)} \mathbb{E}_x^{b]}(q^{T_a}) = \dfrac{H_q^{b]}(x)}{H_q(x)} \dfrac{H_q(a)}{H_q^{b]}(a)} = \dfrac{\kappa_q^{b]}(a)}{\kappa_q^{b]}(x)} .
	\end{align}
	Since $(X,^{H_q}\mP_x^{b]})$ is a transient upward skip-free Markov chain, using Lemma \ref{thm:hit}  and  Theorem \ref{thm:green2}\eqref{it:mb} for $q=1$, one easily deduces, with the obvious notation, that  $\frac{1}{C \kappa_q^{b]}(x)}=  \:\: \:^{H_q}K^{b]}\delta_\r(x), $ for some $C>0$.  
Thus the mapping $x \mapsto \kappa_q^{b]}(x)$ is non-increasing.  Henceforth
	$\kappa_q^{b]} = \lim_{x \to \r} \kappa_q^{b]}(x) \leq \kappa_q^{b]}(b+1) < \infty$. Observing that both $H_q(\r) $ and $H_q^{b]}(\r)$ are finite  when $X \in \Sf \setminus \Sf_{\infty}$, we readily get that  $\kappa_q^{b]}>0$ which completes the proof in this case.  It remains to show that
	\begin{align}\label{eq:Kqbfinite}
	\lim_{a \rightarrow \r} \dfrac{\kappa_q^{b]}(a)}{\kappa_q^{b]}(x)} &=  {\lim_{a \rightarrow \r}} ^{H_q}\mathbb{P}_x^{b]}(T_a < \infty) > 0,
	\end{align}
when $X \in  \Sf_{\infty}$.
To this end, we   assume the contrary, that is,  $\lim_{a \rightarrow \r} {}^{H_q}\mathbb{P}_x^{b]}(T_a < \infty) = 0$. Since ${}^{H_q}\mathbb{P}_x^{b]}(T_a < \infty) = \mathbb{P}_x^{H_q}(T_a < T_{b]})$, the assumption becomes $\lim_{a \rightarrow \r} \mathbb{P}_x^{H_q}(T_a < T_{b]}) = 0$ and \eqref{eq:trick}  leads to
	\begin{align}\label{eq:assump}
	\lim_{a \rightarrow \r} \mathbb{P}_x^{H_q}(T_a < T_b) = 0.
	\end{align}
Next, by means of a first step analysis and the upward skip-free property, we obtain
	\begin{align} \label{eq:pk}
	\mathbb{P}_b^{H_q}(T_b^+ < T_a) &= p^{H_q}(b,b+1) \mathbb{P}_{b+1}^{H_q}(T_b^+ < T_a) + p^{H_q}(b,b) + \sum_{y < b} p^{H_q}(b,y) \mathbb{P}_{y}^{H_q}(T_b^+ < T_a).
	\end{align}
	Taking $a \to \r$, the left-hand side converges to $\mathbb{P}_b^{H_q}(T_b^+ < \infty)$ (recall that $T_b^+ = \inf\{n \geq 1;\: X_n = b\}$) due to the monotone convergence theorem, the upward skip-free property and  the fact that $X \in \Sf_{\infty}$. On the right-hand side of \eqref{eq:pk}, the first term converges to $p^{H_q}(b,b+1)$ as a result of \eqref{eq:assump}, while the third term converges to $\sum_{y < b} p^{H_q}(b,y) \mathbb{P}_{y}^{H_q}(T_b^+ < \infty)$ by invoking the  dominated convergence theorem. Therefore, we arrive at
	\begin{align*}
	\mathbb{P}_b^{H_q}(T_b^+ < \infty) &= p^{H_q}(b,b+1) + p^{H_q}(b,b) + \sum_{y < b} p^{H_q}(b,y) \mathbb{P}_{y}^{H_q}(T_b^+ < \infty) \\
	&= p^{H_q}(b,b+1) + \sum_{y \leq b} p^{H_q}(b,y) = \mathbb{P}_{b}^{H_q}(\zeta>1)
	= 1,
	\end{align*}
	where the second equality comes from the identity $\mathbb{P}_{y}^{H_q}(T_b^+ < \infty) = 1$ which holds since $\mP_y^{H_q}(X_{\zeta} = \r) = 1$ and $y< b$, while the last equality is due to Theorem \ref{thm:htransp} with the fact that $H_q \in \mathcal{H}_q$ since $X \in \Sf_{\infty}$. This is not possible since $X$ is transient. Therefore, we conclude that $\kappa_q^{b]} > 0$.
\end{proof}


\subsection{Proof of Theorem \ref{thm:green2}\eqref{it:poh}}
We start with the following extension of Lemma \ref{lem: below}.
\begin{lemma}\label{lem: belowf}
For any  $x,y \in E$ and $0<q<1$, we have
	\begin{equation}\label{eq:hitg}
	\mathbb{E}_x\left(q^{T_y}\right)= \dfrac{H_q(x) - \Hqy(x)}{H_q(y)}.
	\end{equation}
\end{lemma}
\begin{proof}
The case when $x\leq y$ is proved in \eqref{eq:uphit}. Next assume that $x>y$. Thanks to the upward skip-free property of $X$, for any $b \in E$ the mapping $\1_{\{T_b < T_a\}}$ is increasing in $a$ large enough. Then, the monotone convergence theorem and the fact that $X \in \Sf_{\infty}$ give $ \lim_{a \rightarrow \r}\mathbb{E}_x(q^{T_b} \1_{\{T_b < T_a\}}) = \mathbb{E}_x(q^{T_b})$. The sought  result follows immediately from Lemma \ref{lem: below} and Lemma \ref{lem:Kqb}.
\end{proof}
We  are now ready to prove the expression \eqref{eq:bargxysum}. First using \eqref{eq:uphit}, Lemma \ref{thm:hit} and the definition of the Martin kernel in \eqref{eq:defmk},  we obtain, for any $x \leq \mfo$,
\begin{align} \label{eq:ho}
\mathbb{E}_x(q^{T_\mfo}) = H_q(x) = \dfrac{G_q(x,\mfo)}{G_q(\mfo,\mfo)}= \dfrac{G_q(x,\mfo)}{C_q}.
\end{align}
Next, for sake of clarity we state the analogue of the identity \eqref{eq:uphit} for  the dual chain  $(X,\widehat{\mP})$.
\begin{lemma}\label{lem:dualminEexnar}
	For all $0 < q < 1$ and any $x \leq y$,
	\begin{align}\label{eq:hatHq}
	\widehat{\mathbb{E}}_y(q^{T_x}) &= \dfrac{\widehat{H}_q(y)}{\widehat{H}_q(x)}.
	\end{align}
\end{lemma}
A specific application of the previous result yields, for any $x \leq \mfo$, that
\begin{align}\label{eq:hatHq2}
\widehat{\mathbb{E}}_\mfo(q^{T_x}) &= \dfrac{1}{\widehat{H}_q(x)} = \dfrac{\widehat{G}_q(\mfo,x)}{\widehat{G}_q(x,x)} = \dfrac{ G_q(x,\mfo)}{G_q(x,x)},
\end{align}
where we use the identity $\widehat{H}_q(\mfo) = 1$ in the first equality, the dual version of \eqref{eq:ho} for the second one and the integrated version of the dual identity \eqref{eq:dualdef} for the last one. Combining \eqref{eq:ho} and \eqref{eq:hatHq2}, we get, for any $x \leq \mfo$,
\begin{align}\label{eq:gqyy}
G_q(x,x) = C_q H_q(x) \widehat{H}_q(x).
\end{align}
For any $x \geq \mfo$, we reverse the role of $x$ and $\mfo$  to obtain, respectively,
\begin{align}
\mathbb{E}_\mfo(q^{T_x}) &= \dfrac{1}{H_q(x)} = \dfrac{G_q(\mfo,x)}{G_q(x,x)}, \label{eq:hatHqrev}\\
\widehat{\mathbb{E}}_x(q^{T_\mfo}) &= \widehat{H}_q(x) = \dfrac{\widehat{G}_q(x,\mfo)}{\widehat{G}_q(\mfo,\mfo)} = \dfrac{G_q(\mfo,x)}{G_q(\mfo,\mfo)} \label{eq:hatHqrev2}.
\end{align}
Combining \eqref{eq:hatHqrev} and \eqref{eq:hatHqrev2}, we again arrive at \eqref{eq:gqyy}, which shows that \eqref{eq:gqyy} holds for all $x \in E$. Note that \eqref{eq:gqyy} holds regardless of the boundary condition at $\r$. In particular, when $X \in \Sf_{\infty}$,  \eqref{eq:hitg}, \eqref{eq:ho} (replacing  $y$ by $\mfo$), \eqref{eq:gqyy} and Lemma \ref{thm:hit} give \eqref{eq:bargxysum} which complete the proof of Theorem \ref{thm:green2}.

\subsection{$X \in \Sf \setminus \Sf_{\infty}$}\label{sec:regular}

Throughout this section we work under the following hypothesis that $X \in \Sf \setminus \Sf_{\infty}$, that is, $\r$ is a regular boundary. The state space $E$ includes the point $\r$, that is, $E = [[\l,\r]$. Under $\mathbb{P}^{[\r}$, $X$ is a skip-free Markov chain on the state space $E^{[\r} = [[\l,\r)$, which is killed whenever the process hits the state $\r$. We are ready to state the following.

\begin{proposition}\label{prop:regular}
Suppose $\r$ is a regular boundary. Then for any $x > b$,
$$\mathbb{E}_x(q^{T_b}) = \begin{cases}
\dfrac{H_q(x) - \bar{K}_q^{b]} H_q^{b]}(x)}{H_q(b)} &\text{if $\r > x > b,$}\\
\dfrac{G_q(\r,b)}{G_q(b,b)} &\text{if $\r = x > b,$}
\end{cases}$$
where $\bar{K}_q^{b]} = \dfrac{G_q^{[\r}(\mfo,\mfo) \widehat{H}_q^{[\r}(b) }{G_q(\mfo,\mfo) \widehat{H}_q(b)}$ and $G_q(\r,b) = \dfrac{\eta_b(q)}{c + (1-q)p + \sum_{j \in E} (1-q) \eta_j(q)},$
where $c \geq 0$, $p > 0$, and $\eta(q) = (\eta_j(q))_{j \in E}$ is a family of non-negative numbers that satisfies, for any $0<s,q<1$,
$$ s \eta(s) - q \eta(q) = (s - q) \eta(q) G_q^{[\r},$$
where $G_q^{[\r} = (G_q^{[\r}(i,j))_{i,j \in E}$.
\end{proposition}

In order to prove this Proposition, we need the following classical result that enables to connect the $q$-potential of $G_q$ and  $G_q^{[a}$ (resp.~ $\widehat{G}_q$ and $\widehat{G}_q^{b]}$).

\begin{lemma}\label{lem:excursion}
	We have, for any $0 < q < 1$,
	\begin{align}	
	G_q(x,y) &= G_q^{[a}(x,y) + \E_x(q^{T_a}) G_q(a,y)  \label{eq:killedabv} , \quad x,y \in E^{[a}, \\	
	\widehat{G}_q(x,y) &= \widehat{G}_q^{b]}(x,y) + \widehat{\E}_x(q^{T_b}) \widehat{G}_q(b,y), \quad x,y \in E^{b]}  \label{eq:killedbel}.
	\end{align}
\end{lemma}

\begin{proof}
 Since $X$ is upward skip-free, $T_{[a} = T_a$ and  for any $n \in \mathbb{N}$, $a \in E$, $x,y \in E^{[a}$, we have
	\begin{align*}
	\mathbb{P}_x(X_n = y) &= \mathbb{P}_x^{[a}(X_n = y) + \sum_{k = 1}^{n-1} \mathbb{P}_x(X_n = y | T_a = k) \mathbb{P}_x(T_a = k) \\
	&= \mathbb{P}_x^{[a}(X_n = y) + \sum_{k = 1}^{n-1} \mathbb{P}_a(X_{n-k} = y) \mathbb{P}_x(T_a = k),
	\end{align*}
	where the second equality follows from strong Markov property. Next, multiplying by $q^n$, dividing by $\pi(y)$, which is positive by irreducibility, and summing over $n$, we obtain
	\begin{align*}
	G_q(x,y) &= G_q^{[a}(x,y) + \sum_{n=1}^{\infty} \sum_{k = 1}^{n-1} q^{n-k} \frac{\mathbb{P}_a(X_{n-k} = y)}{\pi(y)} q^k \mathbb{P}_x(T_a = k) \\
	&= G_q^{[a}(x,y) + \sum_{k=1}^{\infty} \sum_{n = k+1}^{\infty} q^{n-k} \frac{\mathbb{P}_a(X_{n-k} = y)}{\pi(y)} q^k \mathbb{P}_x(T_a = k) \\
	&= G_q^{[a}(x,y) + \E_x(q^{T_a}) G_q(a,y)
	\end{align*}
which proves  \eqref{eq:killedabv}.
	\eqref{eq:killedbel} is the dual statement of \eqref{eq:killedabv}.
\end{proof}
We are  now ready to complete the proof of Proposition \ref{prop:regular}.
First, observe that  $(X,\mP^{[\r}) \in \Sf_{\infty}$ and thus   by means of  Lemma \ref{lem:excursion} and \eqref{eq:gqyy}, one gets
\begin{align}\label{eq:gqbb2} 
	G_q(b,b) &= G_q^{[\r}(\mfo,\mfo)  H_q^{[\r}(b) \widehat{H}_q^{[\r}(b) + \dfrac{H_q(b)}{H_q(\r)} G_q(\r,b).
\end{align}
 If one chooses the same reference point $\mfo$ for $(X,\mP)$ and $(X,\mP^{[\r})$, then plainly, for any  $x \in E^{[\r}$, $H_q(x) = H_q^{[\r}(x)$,  this leads, using again \eqref{eq:gqyy}, to
\begin{align}\label{eq:rreg}
	G_q(\mfo,\mfo) \widehat{H}_q(b) = G_q^{[\r}(\mfo,\mfo) \widehat{H}_q^{[\r}(b) + \dfrac{G_q(\r,b)}{H_q(\r)}.
\end{align}
For $\r > x > b$, using the lemmas \ref{thm:hit} and \ref{lem:excursion}, we have
\begin{align*}
	\mathbb{E}_x(q^{T_b}) &= \dfrac{G_q(x,b)}{G_q(b,b)}
						  = \dfrac{G_q^{[\r}(x,b) + \E_x(q^{T_{\r}}) G_q(\r,b)}{G_q(\mfo,\mfo) H_q(b) \widehat{H}_q(b)} \\
						  &= \dfrac{G_q^{[\r}(\mfo,\mfo)(H_q^{[\r}(x) - K_q^{b]} H_q^{b]}(x)) \widehat{H}_q^{[\r}(b) + \frac{H_q(x)}{H_q(\r)} G_q(\r,b)}{G_q(\mfo,\mfo) H_q(b) \widehat{H}_q(b)} \\
						  &= \dfrac{G_q(\mfo,\mfo) H_q(x) \widehat{H}_q(b) - G_q^{[\r}(\mfo,\mfo) H_q^{b]}(x) \widehat{H}_q^{[\r}(b) }{G_q(\mfo,\mfo) H_q(b) \widehat{H}_q(b)} \\
						  &= \dfrac{H_q(x) - \bar{K}_q^{b]} H_q^{b]}(x)}{H_q(b)},
\end{align*}
where we used \eqref{eq:rreg} in the fourth equality.
For $x=\r  > b$, it follows from \eqref{eq:gqyy} that
\begin{align*}
	\mathbb{E}_{\r}(q^{T_b}) &= \dfrac{G_q(\r,b)}{G_q(b,b)}
						  = \dfrac{G_q(\r,b)}{G_q(\mfo,\mfo) H_q(b) \widehat{H}_q(b)}.
\end{align*}
To complete the proof one appeals to  Theorem 3.1 of \cite{R59} that enables the determination of $G_q(\r,b)$.

\section{Fluctuation identities}  \label{sec:fluc}

We pursue our program by exploiting the potential theoretic results of the previous Section to derive the fluctuation identities of a general skip-free Markov chain. This consists in determining the law, through the pgf,  of the first exit time of $X$ to a (in)finite interval.  These quantities are critical in many applications in various settings including insurance mathematics, biology and epidemiology to name but a few. For instance, the time $T_{0]}$ corresponds to the time of ruin for a risk process  having only negative jumps, which is a natural assumption in risk theory as they (the jumps) model the size of claims. We emphasize that our methodology enables the determination of the pgf of this time of ruin and hence allows the solution of this problem for general Markovian risk processes.   We also recall that for  a birth-death chain the pgf of exit times is given as a linear combination of the two F$q$E functions, the one for $(X,\mP)$ and its dual $(X,\widehat{\mP})$. There exists also a theoretical characterization of the laws of these variables for random walks through the celebrated Wiener-Hopf factorization.  This technique, which finds its root in  complex analysis,  has been nicely exploited by pure probabilistic arguments, see Feller \cite{Fe71} and Greenwood and Pitman \cite{GP80}. Our original approach appears to be more comprehensive for this issue in the context of general skip-free Markov chains. We have already mentioned that in order to illustrate its applicability, we will provide in the next subsection an alternative way to recover the fluctuation identities for skip-free random walks. The analogue of these results for skip-free continuous-time Markov processes on the real line can be found in \cite{Patie-Vigon} and applications to generalized Ornstein-Uhlenbeck processes and continuous branching processes with immigration are carried out in \cite{Lefevre2013} and \cite{Patie-Wev} respectively. We focus  below on the case $X \in \Sf_{\infty}$, the other case can be dealt with similarly by means of Proposition  \ref{prop:regular}.

\begin{theorem}\label{thm:minEexnad}
	Suppose that $X \in \Sf_{\infty}$. Then for any $b \in E$, $x \in E^{b]}$ and $0<q\leq 1$, we have
	\begin{equation}\label{eq:fptb}
	\mathbb{E}_x(q^{T_{b]}}) = 1 + (q-1) C_q \sum_{y \in E^{b]}}  \widehat{H}_q(y) \left(\Hqb(x) - \Hqy(x) \right) \pi(y), 
	\end{equation}
	and, for any $a>b$, $x \in E^{(b,a)^c}$,
\begin{equation}\label{eq:fptab}
	\mathbb{E}_x(q^{T_{(b,a)^c}}) = 1 + (q-1) C_q \sum_{y \in E^{(b,a)^c}}  \widehat{H}_q(y) \left( \dfrac{ \Hqy(a)}{\mathbf{H}_q^{b]}(a)} \Hqb(x) - \Hqy(x)  \right) \pi(y),
	\end{equation}
	where we recall that $C_q = G_q(\mfo,\mfo)$ as in Theorem \ref{thm:green2} and for birth-death chains $C_q$ is the inverse of the Wronskian.
\end{theorem}
\begin{rk} \label{rk:ft}
	We emphasize that the expressions   \eqref{eq:fptb} and \eqref{eq:fptab} are comprehensive and are very useful for solving the first exit times problems when applied to some specific instances. Indeed they reveal that the characterization of the (probability generating function of the) distribution  of first passage times of skip-free Markov chains   boils down to determining the positive constant $C_q$ and the functions $\widehat{H}_q$ and $\Hqb$.
In practice, it turns out that the knowledge of a transformation, such as the Laplace transform, Fourier transform  or moment generating function, of the one-step transition (and hence by integration of the $q$-potential) of the chain enables the determination of $H_q$ and $\widehat{H}_q$. The constant $C_q$ can be determined by an argument of analytical continuation applied to the transform of the $q$-resolvent. Finally, (a transformation of) the function $\Hqb$ can be obtained from the previous identifications combined with the following relations
\begin{equation*}
		\Hqy(x)=\frac{1}{C_q \: \widehat{H}_q(y)}
		\left(H_q(x) -  G_q(x,y)\right).
		\end{equation*}
that can be easily derived  from \eqref{eq:bargxysum}. This procedure will be illustrated in Section \ref{sec:sfrw} below to the class of skip-free random walks and, in the subsequent paper \cite{Patie-Wang}, to  the (continuous-time) branching Galton-Watson processes with immigration.
\end{rk}
\begin{rk}
	The analysis in Theorem \ref{thm:minEexnad} can be extended to study the joint law of $X_{T_{b]} - 1}$ and $X_{T_{b]}}$. Indeed, for $x,y \in E$ and $k \leq b$,
	\begin{align*}
	\E_x(q^{T_{b]}} \1_{\{X_{T_{b]} - 1} = y, X_{T_{b]}} = k\}}) &= \sum_{n=1}^\infty q^n \mP_x(X_{n-1} = y, X_n = k, T_{b]} = n) \\
	&= q \: G^{b]}_q(x,y)  p(y,k),
	\end{align*}
	where the last equality follows from the Markov property.
	Taking $q \rightarrow 1^-$ and using the  monotone convergence theorem, one gets
	\begin{align*}
	\mP_x(X_{T_{b]} - 1} = y, X_{T_{b]}} = k) &= G^{b]}(x,y)p(y,k).
	\end{align*}
	Summing over $y \in E$ (or $y \in E^{b]}$, since $G^{b]}(x,y) = 0$ for $y \leq b$), we have
	$$\mP_x(X_{T_{b]}} = k) = \sum_{y \in E} G^{b]}(x,y)p(y,k).$$
\end{rk}

We proceed with the proof of Theorem \ref{thm:minEexnad}.
First, let $B \subset E$ and denote $G_q^{B}$ (resp.~$\widehat{G}_q^{B}$) to be the $q$-potential of the $X$ (resp.~its dual $\widehat{X}$) killed upon entering into the set  $B$. We  recall the Hunt's switching identity for Markov chains, which can be found in \cite[page $140$]{KSK}, and says that, for any $x,y \in E\backslash B$,
\begin{equation}\label{eq:huntdual}
	G_q^B(x,y) = \widehat{G}_q^{B}(y,x) .
\end{equation}
For sake of simplicity,  we simply write $G_q^{B} = G_q^{b]}$ (resp.~$G_q^{A} = G_q^{[a}$)  if $B = (\l,b]$ (resp.~if $B = [a,\r)$). With this notation in mind, we express the $q$-potential kernels of $(X,\mP^{[a})$ and $(X,\mP^{b]})$ in terms of F$q$E functions of the three processes $(X,\mP)$, $(X,\widehat{\mP})$ and $(X,\mP^{y]})$.
\begin{lemma}\label{lem:green2}
	Suppose that $X \in \Sf_{\infty}$.
	\begin{align}
	G_q^{[a}(x,y) &= C_q \widehat{H}_q(y) \left(\dfrac{ \Hqy(a)}{H_q(a)} H_q(x) - \Hqy(x)  \right), \quad x,y \in E^{[a}, \label{eq:killedabvH} \\
	G_q^{b]}(x,y) &=  C_q\widehat{H}_q(y) \left(\Hqb(x) - \Hqy(x)  \right), \quad x,y \in E^{b]}, \label{eq:killedbelH} \\
	G_q^{(b,a)^c}(x,y) &= C_q \widehat{H}_q(y) \left( \dfrac{ \Hqy(a)}{\Hqb(a)} \Hqb(x) -  \Hqy(x) \right), \quad x,y \in E^{(b,a)^c}. \label{eq:killed2H}
	\end{align}
\end{lemma}

\begin{proof}

To  show \eqref{eq:killedabvH}, we  use \eqref{eq:killedabv} and \eqref{eq:bargxysum} to get
\begin{align*}
	G_q^{[a}(x,y) &= G_q(x,y) - \E_x(q^{T_a}) G_q(a,y) \\
	&= C_q (H_q(x) - \Hqy(x) ) \widehat{H}_q(y) - \dfrac{H_q(x)}{H_q(a)} C_q (H_q(a) - \Hqy(a))\widehat{H}_q(y) \\
	&= C_q \widehat{H}_q(y) \left(\dfrac{H_q(x) \Hqy(a)}{H_q(a)} - \Hqy(x) \right).
\end{align*}
Next, using  the Hunt's switching identity \eqref{eq:huntdual} and \eqref{eq:killedbel}, we have
\begin{align*}
	G_q^{b]}(x,y) &= \widehat{G}_q^{b]}(y,x) = \widehat{G}_q(y,x) - \widehat{\E}_y(q^{T_b}) \widehat{G}_q(b,x) = G_q(x,y) - \dfrac{\widehat{H}_q(y)}{\widehat{H}_q(b)} G_q(x,b) \\
	&= C_q (H_q(x) - \Hqy(x)  ) \widehat{H}_q(y) - \dfrac{\widehat{H}_q(y)}{\widehat{H}_q(b)} C_q (H_q(x) - \Hqb(x) )\widehat{H}_q(b) \\
	&= C_q\widehat{H}_q(y) \left(\Hqb(x) - \Hqy(x)  \right),
\end{align*}
which proves \eqref{eq:killedbelH}. Finally, to get \eqref{eq:killed2H}, we use the identity
$G_q^{b]}(x,y) = G_q^{(b,a)^c}(x,y) + \E_x^{b]}(q^{T_a}) G_q^{b]}(a,y)$ and \eqref{eq:killedbelH} combined with Lemma \ref{lem:minEb}.
\end{proof}
With Lemma \ref{lem:green2} in mind, the proof of  Theorem \ref{thm:minEexnad} follows readily by applying the  second claim stated in the following classical results which we prove for sake of completeness.
\begin{lemma}
For any $b<x$ and $n\in \N$, we have
\begin{equation}
\mP_x( T_{b]} > n)=P_n^{b]}\1(x)=\frac{\pi \widehat{P}_n^{b]}\delta_x}{\pi(x)}
\end{equation}
and thus, for any $0<q<1$,
\begin{equation}
	\mathbb{E}_x(q^{T_{b]}})= 1 + (q-1){\bf{G}}_q^{b]}\1(x) 	= 1 + (q-1) \frac{\pi \widehat{\bf{G}}_q^{b]}\delta_x}{\pi(x)},
\end{equation}
where ${\bf{G}}_q^{b]}f(x) = \sum_{y \in E^{b]}}  f(y) G^{b]}_q(x,y)\pi(y)$.
\end{lemma}
\begin{proof}
First, an application of Fubini's theorem yields, that for any  $b<x$ and $0<q\leq 1$,
\begin{align*}
{\bf{G}}_q^{b]}\1(x) &=	\sum_{y \in E^{b]}}  G^{b]}_q(x,y)\pi(y) = \sum_{y \in E^{b]}} \sum_{n = 0}^{\infty} q^n \mP_x(X_n = y, n < T_{b]}) \\ &= \sum_{n = 0}^{\infty} \sum_{y \in E^{b]}} q^n \mP_x(X_n = y, n < T_{b]}) = \sum_{n = 0}^{\infty} q^n \mP_x( T_{b]} > n ).
\end{align*}
Moreover, from the Hunt's switching identity \eqref{eq:huntdual}, we observe that
\begin{align*}
{\bf{G}}_q^{b]}\1(x) &= \sum_{y \in E^{b]}}  G^{b]}_q(x,y)\pi(y) = \sum_{y \in E^{b]}} \pi(y) \widehat{G}^{b]}_q(y,x) = \frac{\pi \widehat{{\bf{G}}}_q^{b]}\delta_x}{\pi(x)}
\end{align*}
which completes the proof of the first claim. Finally, the lemma is proved after observing that, for any $x \in E^{b]}$,
\begin{align*}
	\mathbb{E}_x(q^{T_{b]}}) &= \sum_{n = 1}^\infty q^n \mP_x(T_{b]} = n) = \sum_{n = 1}^\infty q^n \left( \mP_x(T_{b]} > n-1) - \mP_x(T_{b]} > n) \right) \\
	&= q \sum_{n = 0}^{\infty} q^n \mP_x( T_{b]} > n ) - \sum_{n = 1}^\infty q^n  \mP_x(T_{b]} > n) \\
	&= 1 + (q-1) \sum_{n = 0}^{\infty} q^n \mP_x( T_{b]} > n ).
\end{align*}
\end{proof}

\begin{rk}
	Another method to study the first passage time of $b$ is by collapsing and combining all the states at or below $b$, and consider the hitting time to the glued state for the chain.
	Precisely, let $\widetilde{P}$ be the transition matrix of the glued chain with an absorbing boundary at $b$, that is,
	$$\widetilde{P} = \kbordermatrix{%
		& b & b+1 & b+2 & \ldots\\
		b & 1 & 0 & \ldots & \ldots\\
		b+1 & \sum_{j=0}^{b} p(b+1,j)  & p(b+1,b+1) & \ldots & \ldots\\
		b+2 & \sum_{j=0}^{b} p(b+2,j)  & p(b+2,b+1) & \ldots & \ldots\\
		\vdots & \vdots & \vdots & \vdots & \vdots
	}$$	
	By construction,
	$$\mP_x(T_{b]} \leq n) = \widetilde{\mP}_x(T_b \leq n).$$
	Therefore, by \eqref{eq:hitg}, using the obvious notation,
	\begin{align*}
	\mathbb{E}_x(q^{T_{b]}}) = \widetilde{\mathbb{E}}_x(q^{T_{b}}) = \dfrac{\widetilde{H_q}(x) -  \widetilde{H}_q^{b]}(x)}{\widetilde{H}_q(b)}.
	\end{align*}
\end{rk}

\subsection{Skip-free random walk on $\mathbb{Z}$ revisited}\label{sec:sfrw}
 We illustrate how the comprehensive result stated  in Theorem \ref{thm:minEexnad} can be used to obtain explicit representations for the pgf of the first exit time of a skip-free Markov chain when applied to some specific instances. Below,  we consider an upward skip-free random walk $(X,\mP)$ on the state space $\mathbb{Z}$. Note that for these chains, these problems are of course classic and explicitly treated by Feller in \cite{Fe71}. However, we provide here a new methodology to recover the well known  fluctuation identities of these chains. To emphasize the comprehensive aspect of our approach, we mention that it has also been successfully used in \cite{APW18} and \cite{Patie-Wang} to the study the potential and fluctuation theory of Galton-Watson processes with immigration which are (continuous-time) skip-free to the left Markov chains. Suppose now that $X = (X_n)_{n \in \mathbb{N}_0}$ is a upward skip-free random walk given by $\P_x(X_0 = x)=1$, $x\in \mathbb Z$, and $X_n = X_0+ \sum_{i=1}^n S_i$, where $(S_i)_{n \in \mathbb{N}_0}$ are i.i.d.~random variables with common distribution $F$ which is supported on $\{n \in \mathbb{Z}; \:  n \leq 1\}$. We write $\mathbf{F}$ for the pgf of $F$. Thus, writing $p_{s}(x)=s^{x}$, 
 we have
 \begin{equation}\label{eq:mgsmc}
 Pp_{s}(x)= \mathbf{F}(s) p_{s}(x). 
 \end{equation}
  Since $\lim_{s \to \infty } \mathbf{F}(s)=\infty$, one  also sees that the mapping $s \mapsto \mathbf{F}(s) $ is continuous, increasing and convex  on $(1/h(1),\infty)$ where $h(1)\geq 1$ is the largest root of the equation $\mathbf{F}(s)=1$. Note that $1$ is always a root and $h(1)>1$ when $\mathbf{F}'(1^+)<0$ (the right-derivative at $1$). Therefore, $\frac{1}{\mathbf{F}}$ is continuous, decreasing on $(0, h(1))$ and thus has a well-defined inverse $h : (0,1) \to (h(1),\infty)$. Recall also that, in this case, the dual, with respect to the reference measure $\pi \equiv 1$, is $(X,\widehat{\mP}) \overset{d}{=} (-X, \mP)$.

\begin{proposition}
	We take the reference point to be $\mfo = 0$ and let $0<q<1$.
	\begin{enumerate}
		\item $H_q \in \mathcal{H}_q$ and $\widehat{H}_q \in \widehat{\mathcal{H}}_q$, where for $x \in \mathbb{Z}$,
		\begin{align}\label{eq:hq}
		H_q(x) =h(q)^x, \quad \widehat{H}_q(x)  = h(q)^{-x}.
		\end{align}
		In addition,
		\begin{align}\label{eq:Cq}
		C_q = -q \dfrac{ h'(q)}{h(q)}.
		\end{align}
		
		\item For any $x \geq y+1$, we have $\Hqy(x) = h(q)^y\mathbf{H}_q^{0]}(x-y)$ where, for $s \in \mathbb{R}$ such that $\mathbf{F}(\frac{1}{s}) > \frac{1}{q}$,
		\begin{align}
		\sum_{x \in \mathbb{N}} p_{s}(x) \mathbf{H}_q^{0]}(x) = \dfrac{1}{C_q(q \mathbf{F}(1/s) - 1)}.
		\end{align}
		\item \label{it:qp} For any $x,y \in \mathbb{Z}$, we have for any $0<q<1$, $G_q(x,y)=\overline{G}_q(x-y)$ where
\[\overline{G}_q(x)=-q \dfrac{ h'(q)}{h(q)} \left(h(q)^x-\mathbf{H}_q^{0]}(x)\right).\]
		\item For any $x >0$,
		\begin{align}
		\mathbb{E}_x(q^{T_{0]}}) &= 1 + (1-q) C_q\sum_{y=1}^{x-1} \mathbf{H}_q^{0]}(y) + q h'(q)\dfrac{q-1}{h(q)-1} \mathbf{H}_q^{0]}(x).
		\end{align}
		
	\end{enumerate}
\end{proposition}

\begin{proof}
	First, we deduce easily  from \eqref{eq:mgsmc} and the definition of $h$, that, with  $H_q(x) = h(q)^x=p_{h(q)}(x)$, $0<q<1$,
	$$q P H_q (x) = H_q(x)$$
that is $H_q \in \mathcal{H}_q$. Since $\ln h$ is nonnegative,  clearly $H_q$ is increasing. As $(X,\widehat{\mP}) \overset{d}{=} (-X, \mP)$, the first claims follow readily  from Theorem \ref{thm:minEexnad}. Next, by means of Tonelli's theorem, on the one hand for $s \in \mathbb{R}$ such that $\left|q \mathbf{F}(1/s)\right| < 1$, we have
	$$\sum_{x \in \mathbb{Z}} s^{-x} G_q(0,x) = \sum_{n \geq 0} q^n \E_0 (s^{-X_n}) = \sum_{n \geq 0} \left( q \mathbf{F}(1/s) \right)^n = \dfrac{1}{1 - q \mathbf{F}(1/s)}.$$
	On the other hand, using the translation invariant property of $X$, we have $G_q(-x,0) = G_q(0,x)$, which leads to
	$$\sum_{x \in \mathbb{Z}} s^{-x} G_q(-x,0) = \sum_{x > 0} s^x G_q(x,0) + \sum_{x \geq 0} s^{-x} G_q(-x,0) .$$
	Rearranging the terms and using \eqref{eq:bargxysum} and \eqref{eq:hq} with $\sum_{x > 0} s^x H_q(x) = - \frac{s h(q)}{s h(q)-1}$ yields
	\begin{align}\label{eq:fluct}
	\sum_{x > 0} s^x (H_q(x) - \mathbf{H}_q^{0]}(x)) = \dfrac{1}{C_q(1 - q \mathbf{F}(1/s))} - \dfrac{s h(q)}{s h(q)-1}.
	\end{align}
	As the left-hand side can be treated as $\sum_{x > 0} s^x \E_x(q^{T_0}) \leq \sum_{x > 0} s^x < \infty$ for $0 < s < 1$, it is analytical on the unit disc and by the principle of analytical continuation, one gets that
	\begin{align*}
	C_q &= \lim_{s \to \frac{1}{h(q)}} \dfrac{1 - 1/s h(q)}{1-q\mathbf{F}(1/s)} = \dfrac{-q h'(q)}{h(q)},
	\end{align*}
	which shows \eqref{eq:Cq}. Next,  following \eqref{eq:fluct} and using again $\sum_{x > 0} s^x H_q(x) = - \frac{s h(q)}{s h(q)-1}$, we get
	$$\sum_{x \in \mathbb{N}} s^x \mathbf{H}_q^{0]}(x) = \dfrac{1}{C_q \left(q \mathbf{F}(1/s) - 1\right)}.$$
	Then, note that by the translation invariance property of $X$, $G_q(x,y)=G_q(x-y,0)$ for any $x,y \in E$ which after  some easy algebra yields  $\Hqy(x) = h(q)^y\mathbf{H}_q^{0]}(x-y)$  for any $x \geq y$. The proof of item \eqref{it:qp} follows readily from the previous items and the identity \eqref{eq:bargxysum}.  
	Finally, using the first claim of Theorem \ref{thm:minEexnad}, we have
	\begin{align*}
	\mathbb{E}_x(q^{T_{0]}}) &= 1 + (q-1)C_q \sum_{y \in E^{0]}} \widehat{H}_q(y) \left(\mathbf{H}_q^{0]}(x) - \Hqy(x)  \right) \\
	&= 1 + (q-1)C_q \sum_{y \in E^{0]}} h(q)^{-y} \left( \mathbf{H}_q^{0]}(x) - h(q)^y\mathbf{H}_q^{0]}(x-y) \right) \\
	&= 1 + (1-q) C_q\sum_{y=1}^{x-1} \mathbf{H}_q^{0]}(y) + q h'(q)\dfrac{q-1}{h(q)-1}\mathbf{H}_q^{0]}(x)
	\end{align*}
which completes the proof of the Theorem.
\end{proof}

\section{Classes of first passage time distributions}\label{sec:fpt}

The aim of this part is two-fold.  First, we  provide a characterization of  the first passage times distribution, both from above and below the starting point, in which we built upon the results of Fill \cite{Fill}. This follows from the line of work by Kent and Longford \cite{KL83} who characterize the class of upward and downward hitting time of birth-death Markov chains on non-negative integers. More specifically, they define and introduce a new class $\mathcal{K}(b,\tau,M)$ of discrete infinitely divisible distributions with support on non-negative integers and pgf given by
$$\phi(z) = \exp\bigg\{b(z-1) - \tau + z(z-1) \int_{-1}^{1} (1-pz)^{-1} M(dp)\bigg\}$$
that contain both the upward and downward hitting time of such chain, where $b, \tau$ are non-negative parameters and $M$ is a finite measure on $(-1,1)$ such that
$\int_{-1}^{1} (1+p)^{-1} M(dp) < \infty.$
In particular, using the interlacing property of the eigenvalues for birth-death process, Kent and Longford \cite{KL83} show that the measure $M$ that corresponds to the hitting times are non-negative. However, this interlacing property is lost in general when we move to the realm of skip-free chains. This motivates us to define a more general class of distribution that we call $\mathbb{G}_p$ in Definition \ref{def:classdist} which contains both upward and downward hitting times in Theorem \ref{thm:hitting}. This will be further demonstrated in Remark \ref{rk:KL83} below. Second, we derive an explicit representation of the F$q$E functions associated to $(X,\mP)$ and $(X,\mP^{b]})$ in Section \ref{subsec:fqe}. As we will elaborate in Section \ref{subsec:id}, this allows us to investigate the infinite divisibility of the upward hitting time for skip-free Markov chains and characterizes the associated $R$-function by means of the eigenvalues. In this vein we emphasize that Feller \cite{F66} studied the infinite divisibility of the hitting times of birth-death random walk and Viskov \cite{Viskov00} extended his work to skip-free random walk using the classical Lagrange inversion formula. We focus on
 $\Sf_\l$ the subclass of $\Sf$ that has a finite left boundary point, i.e.~$\l<\infty$.  We shall also need the following definition where we use the notation, for $d\in \mathbb{N}$, $\mathbb{D}^d=\{\boldsymbol{\lambda} = (\lambda_k)_{k=1}^d \in \C^d;\: |\lambda_k| \in [0,1], \:k=1,\ldots,d \textrm{ and } \lambda_i \neq \lambda_k, i \neq k\}$.
\begin{definition}\label{def:classdist}
	Let $p,d\in \N$ with $d \leq p$, $\boldsymbol{\lambda} = (\lambda_k) \in \mathbb{D}^d$ and  $\boldsymbol{m} = (m_k) \in \N^d$ with $\sum_{k=1}^d m_k = p$, and for each $k \in [1,d]$, the multiplicity of $\lambda_k$ is $m_k$, $\mathbf{c}_k = (c_{k,i})_{i=1}^{m_k} \in \C^{m_k}$ and we write $\mathbf{c} = (\mathbf{c}_{k,i})$. We say that a non-negative discrete random variable $X \in \mathbb{G}_p(\boldsymbol{c};\boldsymbol{\lambda};\boldsymbol{m})$ 
if its probability mass function can be written, for any $n\in \N_0$, as
	\begin{equation}
	0\leq \mP(X=n) = \sum_{k=1}^d \sum_{i=1}^{m_k} c_{k,i} \binom{n+i-1}{n} \lambda_k^{n} (1-\lambda_k)^{i}\leq 1,
	\end{equation}
	where
	$\sum_{k=1}^d \sum_{i=1}^{m_k} c_{k,i} \leq 1.$
	In particular, if $\boldsymbol{\theta}=\boldsymbol{\lambda} \in [0,1]^d$, then we write $\mathcal{G}_p(\boldsymbol{c};\boldsymbol{\theta};\boldsymbol{m})=\mathbb{G}_p(\boldsymbol{c};\boldsymbol{\theta};\boldsymbol{m})$. Note that it can be interpreted as a (signed) mixture of geometric random variables when $m_k = 1$ and $\lambda_k \in [0,1]$ for all $k$.	
\end{definition}

Before stating the main result of this part, we introduce the notation   $\boldsymbol{\lambda}^{[x,y]} = (\lambda_i^{[x,y]})_{i=0}^{y-x}$ for the (non-unit) eigenvalues of the transition matrix $P$ restricted to $[x,y]$ with $x \leq y$ and we use $\mathcal{ID}$ to denote the class of infinitely divisible laws, see Section \ref{subsec:id} for formal definition and further discussion on related topics. In Fill \cite{Fill}, under the condition that the transition matrix of an upward skip-free chain has only real and non-negative eigenvalues, it is shown that its upward first hitting time is a convolution of geometric distributions. We also refer to Miclo \cite{Miclo} for similar results in the context of reversible Markov chains.
\begin{theorem}\label{thm:hitting}
	Suppose that $(X,\mP)$ is irreducible. For any $\l \leq b \leq x \leq a \leq \r$, we have, under $\mP_x$,
	\[ T_a - (a - x) \in \mathbb{G}_{a}(\mathbf{c}(x,a);\boldsymbol{\lambda};\boldsymbol{m})\cap \mathcal{ID} \quad \textrm{and} \quad T_b \:(\textrm{resp.~}T_{b]})  \in \mathbb{G}_{\r-1}(\boldsymbol{\tilde{c}}(x,b);\boldsymbol{\tilde{\lambda}};\boldsymbol{\tilde{m}}), \]
	where $\boldsymbol{\lambda}$ are the distinct eigenvalues of $\boldsymbol{\lambda}^{[0,a-1]}$, $\boldsymbol{\tilde{\lambda}}$ are the distinct eigenvalues of $\boldsymbol{\lambda}^{[0,b-1]}$ and $\boldsymbol{\lambda}^{[b+1,\r-1]}$,  $\boldsymbol{m}$ (resp.~$\boldsymbol{\tilde{m}}$) are the multiplicities of $\boldsymbol{\lambda}$ (resp.~$\boldsymbol{\tilde{\lambda}}$). In particular, when $\boldsymbol{\lambda}, \boldsymbol{\tilde{\lambda}}$ are all real and non-negative, then, under $\mP_x$,
		\[ T_a - (a - x) \in \mathcal{G}_{a}(\mathbf{c}(x,a);\boldsymbol{\lambda};\boldsymbol{m})\cap \mathcal{ID} \quad \textrm{and} \quad T_b \:(\textrm{resp.~}T_{b]})  \in \mathcal{G}_{\r-1}(\boldsymbol{\tilde{c}}(x,b);\boldsymbol{\tilde{\lambda}};\boldsymbol{\tilde{m}}). \]
\end{theorem}

\begin{rk}
	Note that for class $\mathcal{S}$ to be introduced in Section \ref{sec:st}, the passage time distributions are in the particular cases of Theorem \ref{thm:hitting}.
\end{rk}

\begin{rk}
	Complex $c_{k,i}$ can occur as a result of partial fraction, and therefore it is necessary for us to consider such cases in Definition \ref{def:classdist}. For example, if the $\alpha_j'$ s in \eqref{eq:fptchar2} are complex (here $\alpha_j$ is usually $\lambda_j^{[0,x-1]}$, so this will occur when $P$ restricted to $[0,x-1]$ has complex eigenvalues), then $c_{k,i}$ are complex in general.
\end{rk}


\subsection{Characterizations of the class $\mathbb{G}_p$}

We first provide a characterization of the class $\mathbb{G}_p$, which will simplify our proof of Theorem \ref{thm:hitting}.

\begin{proposition}\label{prop:fptchar}
	$X \in \mathbb{G}_p(\boldsymbol{c};\boldsymbol{\lambda};\boldsymbol{m})$ if and only if its probability generating function can be written as
	\begin{equation}\label{eq:fptchar}
	\E[q^X] = \sum_{k=1}^{d} \sum_{i=1}^{m_k} \dfrac{C_{k,i}}{(1 - \lambda_k q)^{i}},
	\end{equation}
	where $(C_{k,i})$ depends on $\boldsymbol{c},\boldsymbol{\lambda}$.
	Moreover, if the probability generating function of $X$ can be written as
	\begin{equation}\label{eq:fptchar2}
	\E[q^X] =
	\prod_{j=0}^{b-1} \frac{(1 - \beta_j) q}{1 - \beta_j q}  \prod_{j=0}^{x-1} \frac{1 - \alpha_j q}{(1 - \alpha_j) q}, \end{equation}
	where $\boldsymbol{\beta} = (\beta_j) \in \mathbb{C}^b$,  $\boldsymbol{\alpha} = (\alpha_j) \in \mathbb{C}^x$ with $|\lambda_j|, |\alpha_j| \in [0,1]$ for all $j$, $\boldsymbol{\lambda}$ are the distinct elements of $\boldsymbol{\beta}$ and $\boldsymbol{m}$ are the multiplicities, then $X \in \mathbb{G}_b(\boldsymbol{c};\boldsymbol{\lambda};\boldsymbol{m})$.
\end{proposition}

\begin{rk}\label{rk:KL83}
	We remark that the class $\mathbb{G}_p$ is more general than the class $\mathcal{K}(b,\tau,M)$ studied by Kent and Longford \cite{KL83}. Indeed, using the expression in \cite[Section $10$]{KL83}, we see that \eqref{eq:fptchar2} can be rewritten as
	\begin{align*}
		q^{x-b} \E [q^X] &= \exp \bigg\{ (q-1) \left(\sum_{j=0}^{b-1} |\log(1-\beta_j)| - \sum_{j=0}^{x-1} |\log(1-\alpha_j)|\right) + q(q-1)\int_{-1}^{1} (1-pz)^{-1} M(dp)\bigg\}, \end{align*}
where	$M(dp) = \left(\sum_{j=0}^{b-1} \1_{(0,\beta_j)} - \sum_{j=0}^{x-1} \1_{(0,\alpha_j)} \right)|p|(1-p)dp$.	
	In the particular case of birth-death chains, we can see that the measure $M$ that corresponds to $(T_{a+1},\P_{a})$ is,  by the reality and the interlacing property of eigenvalues, non-negative. In general, the measure $M$ is a signed measure.
\end{rk}

\begin{proof}
	We first show that if $X \in \mathbb{G}_p(\boldsymbol{c};\boldsymbol{\lambda};\boldsymbol{m})$ then \eqref{eq:fptchar} holds. By writing $C_{k,i} = c_{k,i} (1-\lambda_k)^{i}$, we see that
	\begin{align*}
		\E[q^X] &= \sum_{k=1}^{d} \sum_{i=1}^{m_k} C_{k,i} \sum_{n \geq 0} \binom{n+i-1}{n} (q\lambda_k)^{n} = \sum_{k=1}^{d} \sum_{i=1}^{m_k} \dfrac{C_{k,i}}{(1 - \lambda_k q)^{i}}.
	\end{align*}
	The opposite direction can be shown by differentiating the pgf $n$ times followed by dividing $n!$. Next, we show that pgf of the form \eqref{eq:fptchar2} belongs to the class $\mathbb{G}_p$. By means of partial fraction, we note that \eqref{eq:fptchar2} can be written as
	\begin{align*}
	\E[q^X] &= \sum_{k=0}^{d} \sum_{i=1}^{m_k} \dfrac{C_{k,i}}{(1 - \beta_k q)^{i}},
	\end{align*}
	so by \eqref{eq:fptchar}, $X \in \mathbb{G}_b(\boldsymbol{c};\boldsymbol{\lambda};\boldsymbol{m})$.
\end{proof}

\subsection{F$q$E functions of $(X,\mP)$ and $(X,\mP^{b]})$ for skip-free processes}\label{subsec:fqe}
In this section, we give explicit formulas on the F$q$E functions of $(X,\mP)$ and $(X,\mP^{b]})$. Note that the former case follows from Fill \cite{Fill}, whereas the latter  one is a slight generalization of the first one by  allowing  substochastic transition matrices.

\begin{proposition}\label{prop:sfmc}
	Suppose that $(X,\mP)$ is an irreducible upward skip-free Markov chain on $E$ with $|\l| < \infty$. For $b \in E$, we consider the killed process $(X,\mP^{b]})$ on $E^{b]} = [b+1,\r)$. Let $\lambda_{i} = \lambda_i^{[b+1,x-1]}$ be the eigenvalues of $P$ when restricted to $[b+1,x-1]$. If we take $b+1$ to be the reference point (i.e.~$\mfo^{b]} = b+1$), then the F$q$E function of $(X,\mP^{b]})$ is
	\begin{equation*}
	H_q^{b]}(x) = \begin{cases}
	1, &\text{if $x = b+1,$}\\
	\displaystyle \frac{1}{\mP_{b+1}(T_x < T_{b]})} \prod_{i=0}^{x-b-2} \frac{1 - \lambda_{i} q}{(1 - \lambda_{i}) q}, &\text{if $x \geq b+2,$}
	\end{cases}
	\end{equation*}
	where $\mP_{b+1}(T_x < T_{b]}) = \dfrac{p(b+1,b+2)\ldots p(x-1,x)}{(1-\lambda_{0})\ldots(1-\lambda_{x-b-2})}$.
\end{proposition}

\begin{proof}
 Suppose that $X$ is an upward skip-free Markov chain on $[-1,a]$, where $a$ and $-1$ are absorbing states. Let $P$  be the substochastic transition matrix on $[0,a]$, with $(\lambda_i)_{i=0}^{a-1}$ to be the (non-unit) eigenvalues, where $\lambda_i = \lambda_i^{[0,a-1]}$ for  $i = 0,\ldots,a-1$. We define $\widetilde{P} = (\widetilde{p}(i,j))$ for $i,j = 0,1,\ldots,a$
$$\widetilde{p}(i,j) = \begin{cases} \lambda_i, &\quad j = i, \\
1 - \lambda_i , &\quad j = i + 1, \\
0 , &\quad\mbox{otherwise,} \end{cases} $$
where $\lambda_a = 1$. Next, the so-called spectral polynomial is defined to be
$Q_0 = I$ and
$$Q_k = \dfrac{(P - \lambda_0 I)\ldots(P - \lambda_{k-1}I)}{(1-\lambda_0) \ldots (1-\lambda_{k-1})},\quad k = 1,\ldots,a.$$
Define the link matrix $\Lambda = (\Lambda(i,j))$ to be
$$\Lambda(i,j) = Q_i(0,j),\quad i,j = 0,\ldots,a.$$
Then $\Lambda$ satisfies the following properties:
\begin{enumerate}[(i)]
	\item \label{it:l1} The sum of each row of $\Lambda$ is less than or equal to $1$.
	\item \label{it:l2} $\Lambda$ is lower-triangular with $\Lambda(0,0) = 1$ and
	$$\Lambda(k,k) = \dfrac{p(0,1)\ldots p(k-1,k)}{(1-\lambda_0)\ldots(1-\lambda_{k-1})} \neq 0, \quad k =1,\ldots,a,$$
	and hence $\Lambda$ is invertible.
	\item \label{it:l3} $\Lambda P = \widetilde{P} \Lambda$.
\end{enumerate}
The proof of $\eqref{it:l1},\eqref{it:l2}$ and $\eqref{it:l3}$ above are identical to that of Lemma $2.1$ in \cite{Fill}, except that the sum of each row of $\Lambda$ is now less than or equal to $1$, since each row of $(1-\lambda_i)^{-1}(P - \lambda_i I)$ has that property.
The next two results \eqref{eq:tsupht2} and \eqref{eq:tsupht} are analogous to \cite[Lemma $2.3$ and Theorem $1.2$]{Fill}, and we omit the proofs here as they are almost identical, except now we have $\mP_0(T_a < T_{-1}) = \Lambda(a,a)$.
For $t \in \mathbb{N}_0$,
\begin{equation}\label{eq:tsupht2}
	\mP_0(T_a \leq t) = \mP_0(T_a < T_{-1}) \widetilde{P^t}(0,a),
\end{equation}
where
$$\mP_0(T_a < T_{-1}) = \Lambda(a,a) = \dfrac{p(0,1)\ldots p(a-1,a)}{(1-\lambda_0)\ldots(1-\lambda_{a-1})}.$$
For $q \in [0,1]$,
\begin{equation}\label{eq:tsupht}
	\E^{-1]}_0(q^{T_a}) = \mP_0(T_a < T_{-1}) \prod_{i=0}^{a-1} \frac{(1 - \lambda_i) q}{1 - \lambda_i q}.
\end{equation}
The desired result follows immediately from 
Theorem \ref{thm:green2}, Lemma \ref{lem:minEb} and \eqref{eq:tsupht}.
\end{proof}
With the two previous results in hand, we are now ready to complete the proof of  Theorem \ref{thm:hitting}.
\subsection{Proof of Theorem \ref{thm:hitting}} \label{sec:pfc}

Using the result in \cite[Chapter VII Theorem 2.1]{SVH04}, $(T_a - (a-x),\mP_x)$ is infinitely divisible. It follows from \eqref{eq:uphit} that
	\begin{align*}
	\mathbb{E}_x(q^{T_a - (a-x)}) &= q^{- (a-x)}\dfrac{H_q(x)}{H_q(a)},
	\end{align*}
	and by Proposition \ref{prop:sfmc}, the ratio $q^{- (a-x)}\dfrac{H_q(x)}{H_q(a)}$ is of the form \eqref{eq:fptchar2}. Therefore, by Proposition \ref{prop:fptchar}, $T_a - (a - x) \in \mathbb{G}_{a}(\mathbf{c}(x,a);\boldsymbol{\lambda};\boldsymbol{m})\cap \mathcal{ID}$.
	Next, using Lemma \ref{lem: below}, we have
	$$\mathbb{E}_x(q^{T_b} \1_{\{T_b < T_{\r}\}}) = \dfrac{H_q(x)}{H_q(b)} - \dfrac{H_q(\r)}{H_q(b)} \dfrac{H_q^{b]}(x)}{H_q^{b]}(\r)}.$$
	We substitute the F$q$E functions from Proposition \ref{prop:sfmc} and again recognize that it is of the form \eqref{eq:fptchar2}. By Proposition \ref{prop:fptchar}, $T_b \:(\textrm{resp.~}T_{b]})  \in \mathbb{G}_{\r-1}(\boldsymbol{\tilde{c}}(x,b);\boldsymbol{\tilde{\lambda}};\boldsymbol{\tilde{m}})$.

\subsection{Infinite divisibility and R-functions}\label{subsec:id}
In this subsection, we first review several main results in the study of infinitely divisible distribution, which will be used in analyzing the upward hitting times of $(X, \mP)$. We refer interested readers to \cite{SVH04} for a formal reference in the literature.  $\phi$ is the probability generating function of an infinitely divisible distribution on $\mathbb{Z}$ if and only if for every $n \in \mathbb{N}$ there is a pgf $\phi_n$ such that for $q \in (0,1)$,
\[ \phi(q) = \phi_n(q)^n. \]
We first characterize infinite divisibility via canonical sequence: $\phi$ is the pgf of an infinitely divisible distribution on $\mathbb{N}_0$ with $\phi(0) > 0$ and $\phi(1) = 1$ if and only if $\phi$ has the form
$$\phi(q) = \exp\bigg\lbrace-\sum_{k=0}^\infty \dfrac{r_k}{k+1}(1-q^{k+1})\bigg\rbrace,\quad 0 \leq q \leq 1,$$
where $r_k \geq 0$ for all $k \geq 0$. The canonical sequence $(r_k)_{k \in \mathbb{N}_0}$ is unique. Apart from the canonical representation, an alternative way to characterize an infinitely divisible law is by means of the R-function.  $\phi$ is the pgf of a possibly defective  infinitely divisible distribution on $\mathbb{N}_0$ with $\phi(0) > 0$ and  $\phi(1) = \exp\{-b\} \in (0,1], b\geq 0$ ($b=0$ if the distribution is non-defective) if and only if $\phi$ has the form
\begin{align}\label{eq:rfunctdefect}
\phi(q) = \exp\bigg\lbrace -b - \int_q^1 \mathrm{R}(s) \, ds \bigg\rbrace , \quad 0 \leq q < 1,
\end{align}
where R is an absolutely monotone function on $[0,1)$, that is, $R$ has non-negative derivatives of all order. R is unique and is the generating function of the canonical sequence $(r_k)_{k \in \mathbb{N}_0}$ of $\phi$.

Our first result in this section is a by-product of Theorem \ref{thm:green2}, where we obtain the interesting connection between absolute monotonicity and the F$q$E function $H_q$.

\begin{corollary}\label{it:id}
	The mapping $q \mapsto \dfrac{d}{dq} \log \left(q^{-(a-x)} \dfrac{H_q(x)}{H_q(a)}\right)$ is absolutely monotone if and only if $x \leq a$.	
\end{corollary}

\begin{proof} Assume that $x \leq a$. Then,  from \cite[Chapter VII Theorem 2.1]{SVH04} we have  that $(T_a - (a-x),\mP_x)$ is a positive infinitely divisible variable. Combining this fact with the identity \eqref{eq:uphit} and \eqref{eq:rfunctdefect}, we conclude that $\dfrac{d}{dq} \log \left(q^{-(a-x)} \dfrac{H_q(x)}{H_q(a)}\right)$ is an R-function. Conversely, if $x > a$, then $\dfrac{d}{dq} \log \left(q^{-(x-a)} \dfrac{H_q(a)}{H_q(x)}\right) \geq 0$ is an R-function, so
$\dfrac{d}{dq} \log \left(q^{-(a-x)} \dfrac{H_q(x)}{H_q(a)}\right) = - \dfrac{d}{dq} \log \left(q^{-(x-a)} \dfrac{H_q(a)}{H_q(x)}\right) < 0$ cannot be an R-function which completes the proof of the Corollary.
\end{proof}

In the next two propositions, we proceed by offering explicit spectral formulas of R-functions and canonical sequences associated with the infinitely divisible variables $(T_a - a,\mP_0)$ and $(T_{a}-(a-x),\mP_x)$ for $0 \leq x \leq a$, building upon the results in Section \ref{subsec:fqe}, and thereby extending the work by Feller \cite{F66} for the case of birth-death random walk and Viskov \cite{Viskov00} for skip-free random walk. We start by defining a few notations. For $0 \leq j \leq a-1$, let
$${\textrm{R}}_j(s) = \frac{\lambda_j}{1 - \lambda_j s}.$$
Denote the number of real (resp.~complex) eigenvalues by $N_r^{0 \to a} = |\{j : \lambda_j \,\text{is real} \}|$ (resp.~$N_c^{0 \to a} = |\{j : \lambda_j \,\text{is complex} \}|$).

\begin{proposition}
	Suppose that $(X,\mP)$ is irreducible, $a \in \mathbb{N}_0$, and let $\lambda_j = \lambda_j^{[0,a-1]}$. The ${\mathrm{R}}$-function of $(T_a - a,\mP_0)$ is
	$${\mathrm{R}}^{0 \to a}(s) = \sum_{j=0}^{a-1} {\mathrm{R}}_j(s)$$
	with canonical sequence
	$$r^{0 \to a}_k = \sum_{j=0}^{a-1} \lambda_j^{k+1} = \sum_{j = 1}^{N_r^{0 \to a}} \lambda_j^{k+1} + \sum_{j = 1}^{N_c^{0 \to a}} |\lambda_j|^{k+1} \cos{({(k+1)\mathrm{Arg}\lambda_j})} , \quad k \in \mathbb{N}_0,$$ and constant $b^{0 \to a} = - \ln \mP_{0}(T_a < \infty)$.
\end{proposition}

\begin{proof}
	Using the result in \cite[Chapter VII Theorem 2.1]{SVH04}, $(T_a - a,\mP_0)$ is infinitely divisible. By Proposition \ref{prop:sfmc},
	\begin{align*}
	\mathbb{E}_0(q^{T_a - a}) &= \mP_{0}(T_a < \infty) \prod_{j = 0}^{a-1} \dfrac{(1 - \lambda_j)}{1 - \lambda_j q}
	= \mP_{0}(T_a < \infty) \exp\bigg \lbrace- \int_q^1 \sum_{j=0}^{a-1} {\mathrm{R}}_j(s) \, ds \bigg \rbrace,
	\end{align*}
	where the second equality follows from  identity
	$$\dfrac{1 - \lambda}{1 - \lambda q} = \exp \bigg \lbrace - \int_q^1 \dfrac{\lambda}{1 - \lambda s} \, ds \bigg \rbrace,$$
 valid for $|\lambda|< 1/q$.
	Since complex eigenvalues occur in conjugate pair, we can check that $\sum_{j=0}^{a-1} {\mathrm{R}}_j(s) \in \mathbb{R}$.
\end{proof}

\begin{proposition}\label{prop:Rfunctxa}
	Suppose that $0 \leq x \leq a$. Following \eqref{eq:rfunctdefect}, the ${\mathrm{R}}$-function of $(T_a - (a-x),\mP_x)$ is
	$${\mathrm{R}}^{x \to a}(s) ={\mathrm{R}}^{0 \to a}(s) - {\mathrm{R}}^{0 \to x}(s)$$
	with canonical sequence $$r^{x \to a}_k = r^{0 \to a}_k - r^{0 \to x}_k, \quad k \in \mathbb{N}_0,$$ and constant $b^{x \to a} = - \ln \mP_{x}(T_a < \infty)$.
\end{proposition}

\begin{proof}
	Using the result in \cite[Chapter VII Theorem 2.1]{SVH04}, $(T_a - (a-x),\mP_x)$ is infinitely divisible. Let $\beta_j = \lambda_j^{[0,x-1]}$ for $j = 0,\ldots,x-1$. By Proposition \ref{prop:sfmc},
	\begin{align*}
	\mathbb{E}_x(q^{T_a - (a-x)}) &= \mP_{x}(T_a < \infty) \left(\prod_{j = 0}^{x-1} \dfrac{\frac{(1 - \lambda_j)}{1 - \lambda_j q}}{\frac{(1 - \beta_j)}{1 - \beta_j q}} \right) \left( \prod_{j = x}^{a-1} \frac{(1 - \lambda_j)}{1 - \lambda_j q} \right) \\
	&= \mP_{x}(T_a < \infty) \exp\bigg \lbrace- \int_q^1 {\mathrm{R}}^{0 \to a}(s) - {\mathrm{R}}^{0 \to x}(s) \, ds \bigg \rbrace.
	\end{align*}
\end{proof}

Since $r^{x \to a}_k \geq 0$ for $k \in \N_0$, we immediately obtain the following ordering of eigenvalues.

\begin{corollary}\label{cor:Rfunctxa}
	Suppose that $0 \leq x \leq a$, and let $\beta_j = \lambda_j^{[0,x-1]}$ for $j = 0,\ldots,x-1$. Then $$\sum_{j=0}^{a-1} \lambda_j^{k+1} \geq \sum_{j=0}^{x-1} \beta_j^{k+1} \geq 0, \quad k \in \mathbb{N}_0.$$
\end{corollary}

\begin{rk}
	We can rewrite the results in Proposition \ref{prop:Rfunctxa} and Corollary \ref{cor:Rfunctxa} in terms of trace. Define $P^{[0,a]}$ to be the restriction of $P$ from state $0$ to $a$. For $k \in \mathbb{N}_0$, we observe that $r_k^{0 \to a} = \mathrm{Tr}((P^{[0,a]})^{k+1})$, and so Corollary \ref{cor:Rfunctxa} can be rewritten as
	$$\mathrm{Tr}((P^{a]})^{k+1}) \geq \mathrm{Tr}((P^{x]})^{k+1}).$$
\end{rk}

%

\subsubsection{Upward hitting times}

In this part, we characterize a class of discrete distribution that contains all upward hitting times within the class $\Sf$.

\begin{definition}
	We say that a (possibly defective) distribution lies in class $\mathcal{U}$ if it is discrete with support on $\mathbb{N}_0$, infinitely divisible and each term of the canonical sequence can be written as
	\begin{align}\label{eq:uphitclass}
	\dfrac{r_k}{k+1} = \int_{-1}^{1} x^k w(x) \, dx, k \in \mathbb{N}_0,
	\end{align}
where the mapping $w$ satisfies the integrability condition
	$$\int_{-1}^{1} |w(x)|\, dx < \infty.$$
\end{definition}

\begin{proposition}
	Suppose that $0 \leq x \leq a$. Then, $(T_a - (a-x),\mP_x)$ belongs to the class $\mathcal{U}$.
\end{proposition}

\begin{proof}
	It suffices to show \eqref{eq:uphitclass} for $r_k^{x \to a}$. Define $$w^{0 \to a} = \sum_{j=1}^{N_r^{0 \to a}} \1_{(0,\lambda_j)} + \sum_{j=1}^{N_c^{0 \to a}} \1_{(0,|\lambda_j| \cos{({(k+1)\mathrm{Arg}\lambda_j})}^{1/(k+1)})},$$ then
	\begin{align*}
	\int_{-1}^{1} y^k w^{0 \to a}(y) \, dy &= \dfrac{\sum_{j=0}^{a-1} \lambda_j^{k+1}}{k+1} = \dfrac{r^{0 \to a}_k}{k+1}.
	\end{align*}
	In addition, by the triangle's inequality,
	$$\int_{-1}^{1} |w^{0 \to a}(y)|\, dy \leq \sum_{j=0}^{a-1} \lambda_j < \infty.$$
	Let $w^{x \to a} = w^{0 \to a} - w^{0 \to x}$, we obtain
	\begin{align*}
	\int_{-1}^{1} y^k w^{x \to a}(y) \, dy &= \dfrac{r^{0 \to a}_k - r^{0 \to x}_k}{k+1} = \dfrac{r^{x \to a}_k}{k+1}, \\
	\int_{-1}^{1} |w^{x \to a}(y)|\, dy &\leq r_0^{0 \to a} + r_0^{0 \to x} < \infty,
	\end{align*}	
which completes the proof.
	
\end{proof}

\section{The class of skip-free Markov chains similar to  birth-death chains}  \label{sec:st}
In this Section, we develop an original methodology to obtain the spectral decomposition in Hilbert space of the (transition operator of) Markov chains that  belong to  the class $\S$ a subclass  of $\Sf$ (the class of upward skip-free Markov chains) which is defined in Definition \ref{def:Sclass} below. We recall that  as  the transition operator $P$ of a chain in $\Sf$ is usually non-self-adjoint (non-reversible)  in the weighted Hilbert space \[ \ell^2(\pi)=\bigg\{ f: E \mapsto \mathbb{R}; \:||f||^2_{\pi}=\sum_{x\in E} f^2(x)\pi(x)<\infty\bigg\},\]  where $\pi$ is the reference measure,  there is no spectral theorem available for such bounded linear operator except for normal operators. We already point out that the two subsequent  Sections  contain interesting and substantial applications of this spectral decomposition namely the study of the speed of convergence to equilibrium and the separation cutoff phenomenon. We proceed by  defining an equivalence relation on the set of transition matrices in $\Sf$.

\begin{definition}[Similarity]\label{def:sim}
We say  that the transition matrix $P$ of a Markov chain $X \in \Sf$  is similar to the transition matrix of a Markov chain $Q$ on $E$, and we write  $P \sim  Q$, if there exits  a bounded linear operator $\Lambda:\ell^2(\pi_Q) \to \ell^2(\pi)$ ($\pi_Q$ being the reference measure for $Q$) with bounded inverse such that
\begin{equation}
P \Lambda = \Lambda Q.
\end{equation}
When needed we may write  $P \stackrel{\Lambda}{\sim} Q$ to specify the intertwining kernel. Note that $\sim$ is an equivalence relationship on the set of transition matrices. In the context of Markov chains theory, intertwining between birth-death processes has first been studied in \cite{FD}.
\end{definition}

With Definition \ref{def:sim} in mind, we are now ready to define the $\S$ class.
\begin{definition}[The $\S$ class]\label{def:Sclass}
Suppose that $Q \in \mathcal{B}$, the set of transition matrix $Q$ on $E$ of an irreducible (or with at most one absorbing or entrance state) birth-death chain. The similarity orbit of $Q$ (in $\Sf$) is
\[ \S(Q) = \{P \in \mathcal{M};\:  P \sim Q \}, \]
and the $\S$ class is the union over all possible orbits
\[ \S = \bigcup_{Q \in \mathcal{B}} \S(Q). \]
Similarly, we define $\widehat{\S}(Q)$ and $\widehat{\S}$ by replacing $P$ with $\widehat{P}$ above. Finally, we say that $X \in \S^{M}$ if $X \in \S$ and $(X,\widehat{\mP})$ is stochastically monotone, i.e.~$y\mapsto \widehat{\mP}_y(X_1 \leq x)$ is non-increasing for every fixed $x$.
\end{definition}

\begin{rk}
	It is in general hard to find examples for the $\S$ class since $Q$ is restricted to $\mathcal{B}$ and $P$ is restricted to $\mathcal{M}$. However, if we only require $Q$ to be a reversible Markov chain, there are a few examples that have been studied in the thesis of Choi \cite[Section $3.6$]{ChoiThesis}.
\end{rk}
We remark that $\widehat{\Lambda}:\ell^2(\pi) \to \ell^2(\pi_Q)$, the adjoint of $\Lambda$, is a bounded operator with a bounded inverse as well (see e.g.~\cite[Proposition $2.6$]{Con90}). We write $\norm{\cdot}_{op}$ to be the operator norm, i.e.~$\norm{P}_{op}=\sup_{||f||_{\pi}=1}||Pf||_{\pi}$.


Before stating the main result of this Section, we introduce the following class.
\begin{definition}[The $\Mc$ class]\label{def:Mcclass}
	We say that, for some $\r\geq 3$,  $X \in \Mc_{\r}$ if $(X,\mP) \in \Sf$ with $E=\{0,1\ldots,\r\}$ and for every $x \in [0,\r-1]$, its time-reversal $(X,\widehat{\mP})$  satisfies
	\begin{enumerate}
		\item(stochastic monotonicity) $\widehat{\mP}_{x+1}(X_1 \leq x)\leq \widehat{\mP}_{x}(X_1 \leq x) $,
		\item(strict stochastic monotonicity) $\widehat{\mP}_{x+1}(X_1 \leq x+1) < \widehat{\mP}_{x}(X_1 \leq x+1), \quad x\neq \r-1, $
		\item(restricted upward jump) $\widehat{\mP}_{x+1}(X_1 \leq x+k) = \widehat{\mP}_x(X_1 \leq x+k), \quad  x\neq \r-1, \quad k \in [2,\r-1-x]$.
	\end{enumerate}
	Moreover, we say  $X \in \mathcal{MC}^+_{\r}$ if $X \in \Mc_{\r}$ and
	\begin{enumerate}
		\item[(4)](lazy Siegmund dual) $\widehat{\mP}_{x}(X_1 \geq x+1)+\widehat{\mP}_{x+1}(X_1 = x)\leq \frac{1}{2}, \quad x=0,\ldots,\r-1.$
	\end{enumerate}
	When there is no ambiguity of the state space, we write $\Mc = \Mc_{\r}$ (resp.~$\Mc^+ = \Mc_{\r}^+$).
\end{definition}

\begin{rk}
	For a numerical example illustrating the $\Mc$ class, we refer the interested readers to Section \ref{subsec:numericalex}.
\end{rk}

We proceed by recalling  that $P$ has an $\pi$-dual or time-reversal $\widehat{P}$, that is, for $x,y \in E$,
$$\pi(x) \widehat{p}(x,y) = \pi(y) p(y,x),$$
where $\pi$ is a reference measure for $P$. We equip the Hilbert space $\ell^2(\pi)$ with the usual inner product $\langle\cdot,\cdot\rangle_{\pi}$ defined by
$$\langle f,g \rangle_{\pi} = \sum_{x \in E} f(x) g(x) \pi(x), \quad f,g \in \ell^{2}(\pi).$$
 We also recall that a basis $(f_k)$ of a Hilbert space $\mathcal{H}$ is a Riesz basis if it is obtained from an orthonormal basis $(e_k)$ under a bounded invertible operator $T$, that is, $T e_k = f_k$ for all $k$. It can be shown, see e.g.~\cite[Theorem $9$]{Y01}, that the sequence $(f_k)$ forms a Riesz basis if and only if $(f_k)$ is complete in $\mathcal{H}$ and there exist positive constants $A,B$ such that for arbitrary $n \in \N$ and scalars $c_1,\ldots,c_n$, we have
\begin{equation}\label{eq:Rieszb}
	A \sum_{k=1}^n |c_k|^2 \leq \norm{\sum_{k=1}^n c_k f_k}^2 \leq B \sum_{k=1}^n |c_k|^2.
\end{equation}
If $(g_k)$ is a biorthogonal sequence to $(f_k)$, that is, $\langle f_k, g_m \rangle_{\pi} = \delta_{k,m}$, $k,m \in \N$ and $\delta_{k,m}$ is the Kronecker symbol, then $(g_k)$ also forms a Riesz basis.

\begin{theorem}\label{thm:mcins}
 \begin{enumerate}
 \item \label{it:eig} Assume that $\r < \infty$ and $Q$ is the transition kernel of an irreducible birth-death process, then $P \stackrel{\Lambda}{\sim} Q$ if and only if $P$ has real and distinct eigenvalues.
	\item \label{it:sbd} $(X, \mP) \in \S$ with   $P \stackrel{\Lambda}{\sim} Q$ if and only if $(X, \widehat{\mP}) \in \widehat{\mathcal{S}}$ with $Q \stackrel{\widehat{\Lambda}}{\sim} \widehat{P}$, where $\widehat{\Lambda}$ is the adjoint operator of $\Lambda$.
\end{enumerate}
Let us assume that $(X, \mP) \in \S$ with   $P \stackrel{\Lambda}{\sim} Q$. Then the following holds.
\begin{enumerate}[(a)]
	\item \label{it:h} For any $h \in \mathcal{E}$ then $(X, \mP^h) \in \S$ with    $P^h \stackrel{\Lambda^
h}{\sim} Q$ and $\Lambda^h = D_{h}^{-1}\Lambda $, where $D_{h}$ is a diagonal matrix of $h$.
\item \label{it:normal} If $P$ is normal, i.e.~$P\widehat{P}=\widehat{P}P$, then $P$ is self-adjoint, i.e.~$P=\widehat{P}$.
	\item \label{it:comp} $P$ is compact (resp.~trace class) if and only if $Q$ is.
	\item \label{it:spec} Assume that $P$ is compact then	for  any $f \in \ell^{2}(\pi)$ and $n \in \mathbb{N}$, \[ P^n f = \sum_{k=0}^{\r} \lambda_k^n \langle f, f_k^* \rangle_{\pi} f_k,\]
	where the set $(f_k)_{k=0}^\r$ are real eigenfunctions of $P$ associated to the real eigenvalues $(\lambda_k)_{k=0}^\r$ and form a Riesz basis of $\ell^2(\pi)$, and the set $(f_k^*)_{k=0}^\r$ is the unique Riesz basis biorthogonal to $(f_k)_{k=0}^\r$. In particular, for any $x,y \in E$ and $n \in \mathbb{N}$, the spectral expansion of $P$ is given by
		$$ P^n(x,y) = \sum_{k=0}^{\r} \lambda_k^n f_k(x) f_k^*(y)\pi(y).$$
	\item \label{it:Mc} $\Mc \subseteq \S^M$.
	\item \label{it:spech} Assume  that  the condition of the item \eqref{it:spec} holds. If $\r<\infty$ is absorbing and $\l$ is regular, then  with the same notation as above, we have \[ \mP_x( T_{\r} = n)  = \sum_{k=0}^{\r} \lambda_k^n(1-\lambda_k) \langle \1, f_k^* \rangle_{\pi} f_k(x),\]
and assuming that $\l<\infty$ is absorbing and $\r<\infty$ is regular then
\[ \mP_x( T_{\l)} = n)  = \sum_{k=0}^{\r} \lambda_k^n(1-\lambda_k) \langle \1, f^*_k \rangle_{\pi} f_k(x).\]
	\end{enumerate}
\end{theorem}
\begin{proof}
	First, if $P \stackrel{\Lambda}{\sim} Q$, then $P$ has real and distinct eigenvalues since $Q$ has real and distinct eigenvalues. Conversely, if $P$ has real and distinct eigenvalues, $P$ is diagonalizable, so there exists an invertible $\Lambda$ such that
	$$P = \Lambda D \Lambda^{-1}.$$
	where $D$ is the diagonal matrix storing the eigenvalues of $P$. Given the spectral data $D$, by the inverse spectral theorem, see e.g.~\cite[Section 5.8]{DM76}, one can always construct an ergodic Markov chain with transition matrix $Q$ such that
	$$Q = VDV^{-1}.$$
	Next, we show item \eqref{it:sbd}. If $P \stackrel{\Lambda}{\sim} Q$, then for $f \in \ell^2(\pi_{Q})$ and $g \in \ell^2(\pi)$,
	$$\langle f, \widehat{\Lambda} \widehat{P} g \rangle_{\pi_{Q}} = \langle P \Lambda f, g \rangle_{\pi} = \langle \Lambda Q f, g \rangle_{\pi} = \langle f, Q \widehat{\Lambda} g \rangle_{\pi_{Q}},$$
	which shows that $Q \stackrel{\widehat{\Lambda}}{\sim} \widehat{P}$. The opposite direction can be shown similarly. Item \eqref{it:h} follows directly from
	$$P \Lambda = D_h P^h D_h^{-1} \Lambda =  D_h P^h \Lambda^h = \Lambda Q.$$
For the item \eqref{it:normal}, we recall, from the spectral theorem, that if two normal matrices are similar then they are unitary equivalent that is $\Lambda^{-1}=\widehat{\Lambda}$. Then, the proof of this claim is completed since we easily deduce, from the item \eqref{it:sbd}, that
\begin{equation} \label{eq:interadj}
P=P\Lambda \widehat{\Lambda}=\Lambda \widehat{\Lambda}\widehat{P}=\widehat{P}.
\end{equation}
	Next, we turn to the proof of the item \eqref{it:comp}. If $P$ is compact (resp.~ trace class), then $Q = \Lambda^{-1} P \Lambda$ is compact (resp.~ trace class) since the product of bounded and compact (resp.~ trace class) operator is a compact (resp.~ trace class) operator, see \cite[Proposition $4.2$]{Con90}  (resp.~ see \cite[Page $218$]{RS80}). To show, in the item \eqref{it:spec}, that $(f_k)$ and $(f_m^*)$ are biorthogonal, we note that the fact that  
	$P$ has distinct eigenvalues yields that $\langle f_k , f_m^* \rangle_{\pi} = \delta_{{k, m}}$ for any $k,m$.
	 Next, denote $(g_k)$ to be the (orthogonal) eigenfunctions of $Q$, see e.g.~\cite{LR54}. Since $f_k = \Lambda g_k$ and $\Lambda$ is bounded, $(f_k)$ is complete as $(g_k)$ is a basis. As $\Lambda$ is bounded from above and below, for any $n \in \N$ and arbitrary sequence $(c_k)_{k=1}^{n}$, we have
	\begin{align*}
		A \sum_{k=1}^n |c_k|^2 \leq \norm{\sum_{k=1}^n c_k f_k}^2_{\pi} = \norm{\Lambda \sum_{k=1}^n c_k g_k}^2_{\pi} \leq B \sum_{k=1}^n |c_k|^2,
	\end{align*}
	where we can take $A = \norm{\Lambda^{-1}}^{-2}$ and $B = \norm{\Lambda}^2$, so \eqref{eq:Rieszb} is satisfied. It follows from \cite[Theorem $9$]{Y01} that there exists the sequence $(f_k^*)$ being the unique Riesz basis biorthogonal to $(f_k)_{k=0}^\r$, and, any $f \in \ell^2(\pi)$ can be written as
	$$f = \sum_{k=0}^{\r} c_k f_k,$$
	where $c_k = \langle f , f_k^* \rangle_{\pi}$. Desired result follows by applying $P^n$ to $f$ and using $P^n f_k = \lambda_k^n f_k$. In particular, if we take $f = \delta_y$, the Dirac mass at $y$, and evaluate the resulting expression at $x$, we obtain the spectral expansion of $P$. To show \eqref{it:Mc}, we write  $\widetilde{P}$ the so-called Siegmund dual (or $H_S$-dual) of $\widehat{P}$. That is, $\widetilde{P}^T = H_S^{-1} \widehat{P} H_S$ where  $H_S = (H_S(x,y))_{x,y \in E}$ is defined to be
	$H_S(x,y) = \1_{\{x \leq y\}}$
	and its inverse $H_S^{-1}=(H^{-1}_S(x,y))_{x,y \in E}$ is
	$H_S^{-1}(x,y) = \1_{\{x = y\}} - \1_{\{x = y - 1\}}$, see \cite{Sieg76}. Since $X \in \Mc$, then $\widehat{P}$ is stochastically monotone, hence from \cite[Proposition 4.1]{ASM03}, we have that $\widetilde{P}$ is a sub-Markovian kernel. For $x \in [1,\r-1]$, $\widetilde{p}(x,x+1) = \widehat{p}(x+1,x) > 0$ since $\widehat{P}$ is irreducible and downward skip-free, and $\widetilde{p}(x,y) = 0 \quad \forall y \geq x+2$. We also have $\widetilde{p}(0,1) = \widehat{p}(1,0) > 0$. For $x \in [0,\r-2]$, condition $(2)$ in $\Mc$ gives that $\widetilde{p}(x,x-1) > 0$, while condition $(3)$ in $\Mc$ guarantees that $\widetilde{p}(x,y) = 0$ for each $x \in [0,\r-1]$ and $y \in [0,x-2]$. That is, $\widetilde{P}$ is a (strictly substochastic) irreducible birth-death chain when restricted to the state space $[0,\r-1]$. Denote $\widetilde{P}^{bd}$  the restriction of $\widetilde{P}$ to $[0,\r-1]$. By breaking off the last row and last column of $\widetilde{P}$, we can write
	\begin{align}\label{eq:hatPbd}
	\widetilde{P} = \left( \begin{array}{cc}
	\widetilde{P}^{bd} & \mathbf{v}  \\
	\mathbf{0} & 1  \\
	\end{array} \right) = (H_S^{-1} \widehat{P} H_S)^T ,
	\end{align}
	where $\mathbf{0}$ is a row vector of zero, and $\mathbf{v}$ is a column vector storing $\widetilde{p}(x,\r)$ for $x \in [0,\r-1]$.
	Observing that the last row of $\widetilde{P}$ is zero except the last entry, we have
		$$M^{[ 0,\r-1 ]} = \Lambda^{[0,\r-1]} \widetilde{P}^{bd}.$$
		Note that $\widetilde{P}^{bd}$ is a strictly substochastic matrix with $\r$ being a killing boundary. Denote $\widetilde{T}^{bd}$ to be the lifetime of Markov chain with transition kernel $\widetilde{P}^{bd}$. However,  defining, with the obvious notation, for any $x \in [ 0,\r-1 ]$,
		$$\widetilde h(x) = \mathbb{P}_x(\widetilde{T}^{bd}_{\r -1} < \widetilde{T}^{bd}), $$
		we have, according to Theorem \ref{thm:green2}, that $\widetilde h$ is an harmonic function for $\widetilde{P}^{bd}$, i.e.~$\widetilde{P}^{bd} \widetilde h =\widetilde h$.  Hence, a standard result in Martin boundary theory, see e.g.~Theorem \ref{thm:htransp}, entails that the Markov chain with  transition kernel $Q$, defined on $[ 0,\r-1 ]\times [ 0,\r-1 ]$ by  $Q(x,y) = \frac{\widetilde h(y)}{\widetilde h(x)}\widetilde{P}^{bd}(x,y)$, is an ergodic birth-death chain, which completes the proof. Finally to show item \eqref{it:spech}, after observing that, for any $n \in \mathbb{N} $ and $x \in E$,
\[\P_x(T_\r >n ) = \sum_{y \in E}\P_x(X_n=y,T_\r >n )=P\1(x)\]
the first representation in \eqref{it:spec}  with $f\equiv 1$, $n=1$  and easy algebra yields the desired result. The last claim follows by similar means.
\end{proof}




\section{Convergence to equilibrium}
As a first application of the spectral decomposition stated in Theorem \ref{thm:mcins}, we derive accurate information regarding the speed of convergence to stationarity for ergodic chains in $\S$ in both the Hilbert space topology and in total variation distance. There have been a rich literature devoted to the study of convergence to equilibrium for non-reversible chains, see e.g.~\cite{Fill91, LSC97} and the references therein. In these papers, to overcome the lack of a spectral theory, the authors  resort to reversibilization procedures to extract bounds for the distance to stationarity. Another popular approach for studying the rate of convergence is by coupling techniques \cite{Thorisson00}. Our approach reveals a natural extension to the non-reversible case of the classical spectral gap that appears in the study of reversible chains, see e.g.~\cite{LSC97}. To state our result we need  to introduce some notation. We denote the second largest eigenvalue in modulus (SLEM) or the spectral radius of $P$ in the Hilbert space $\ell^2_0(\pi) = \{ f \in \ell^2(\pi); \:~\langle \1,f \rangle_{\pi} = 0 \}$,  by $\lambda_{*} = \lambda_{*}(P) = \sup\{|\lambda_i|;~ \lambda_i \neq 1\}$, then the \textit{absolute spectral gap} is $\gamma_* = 1 - \lambda_*$. For any two probability measures $\mu, \nu$ on $E$, the total variation distance between $\mu$ and $\nu$ is given by
$$|| \mu - \nu ||_{TV} = \dfrac{1}{2} \sum_{x \in E} |\mu(x) - \nu(x)|.$$
For $n \in \mathbb{N}$, the total variation distance from stationarity of $X$ is
$$d(n) = \max_{x \in E} || \delta_x P^n -  \pi ||_{TV}.$$
For $g \in \ell^2(\pi)$, the mean of $g$ with respect to $\pi$ can be written as $\E_{\pi}(g) = \langle g,\1 \rangle_{\pi}$. Similarly, the variance of $g$ with respect to $\pi$ is $\mathrm{Var}_{\pi}(g) = \langle g,g \rangle_{\pi} - \E_{\pi}^2(g) $.
Finally, we recall that Fill in \cite[Theorem $2.1$]{Fill91} obtained when $\r<\infty$ the following  bound valid for all  $n \in \mathbb{N}_0$
\begin{align}\label{eq:tvfill}
	d(n) \leq \dfrac{\sigma_*(P)^n}{2} \sqrt{\dfrac{1-\pi_{min}}{\pi_{min}}},
\end{align}
where $\pi_{min} = \min\limits_{x \in E} \pi(x)$ and $\sigma_*(P) = \sqrt{\lambda_*(P \widehat{P})}$ is the second largest singular value of $P$. We obtain the following refinement for Markov chains in the class $\S$.

\begin{theorem}\label{thm:spectralexp}
	Let  $X \in  \S $  and $X$ is ergodic with stationary distribution $\pi$ and $P \stackrel{\Lambda}{\sim} Q$.
\begin{enumerate}
\item \label{it:varbd} For any $n \in \mathbb{N}_0$,  we have
		\begin{align} \label{eq:hyp}
			\lambda_*^n\leq \norm{P^n - \pi}_{\ell^2(\pi) \to \ell^2(\pi)} \leq  \sigma^{n}_*(P)\1_{\{n<n^*\}} +\kappa(\Lambda)  \lambda_*^n\1_{\{n\geq n^*\}}
		\end{align}
where $n^* = \lceil \frac{\ln \kappa(\Lambda)}{\ln \sigma_*(P)-\ln\lambda_*} \rceil$ and $\kappa(\Lambda)= \norm{\Lambda}_{op} \: \norm{\Lambda^{-1}}_{op} \geq 1$ is the condition number of $\Lambda$.
\item \label{it:v}
If $\r<\infty$ then $P$ is non-reversible if and only if $\kappa(\Lambda)>1$.

\item \label{it:st} A sufficient condition for which $\lambda_* < \sigma_*(P)$ is given by $\max_{i \in E} p(i,i) > \lambda_*$. In such case, for $n$ large enough, the convergence rate $\lambda_*$ given in item \eqref{it:varbd} is strictly better than the reversibilization rate $\sigma_*(P)$.
\item \label{it:totalvar} Suppose now that $\r < \infty$. Then, for any $n \in \mathbb{N}_0$,
	\begin{align*}
	d(n) \leq \dfrac{\min\left(\sigma^{n}_*(P),{\kappa(\Lambda)}  \lambda_*^n \right)}{2} \sqrt{\dfrac{1-\pi_{min}}{\pi_{min}}},
	\end{align*}
	where $\lambda_{*}\leq \sigma_*(P)$.
\end{enumerate}
\end{theorem}
\begin{rk}
 It is interesting to recall that when $P$ is reversible and compact then the sequence of eigenfunctions is orthonormal and  thus an application of the Parseval identity yields $\norm{P^n - \pi}_{\ell^2(\pi) \to \ell^2(\pi)} =  \lambda_*^n$ and $\kappa(\Lambda) = 1$ which  is a specific instance of  item \eqref{it:varbd} and \eqref{it:v}.
\end{rk}
\begin{rk}
  We also recall the discrete analogue of the notion of hypocoercivity introduced in \cite{V09}, i.e.~there exists a constant $C < \infty$ and $\rho \in (0,1)$ such that, for all $ n \in \mathbb{N}$,
 $$\norm{P^n - \pi}_{\ell^2(\pi) \to \ell^2(\pi)} \leq C \rho^n.$$
Note that, in general, these constants are not known explicitly.  We observe that the upper bound in \eqref{eq:hyp} reveals that the ergodic chains in $\S$ satisfy this hypocoercivity phenomena. More interestingly, our approach based on the similarity concept enables us to get on the one hand an explicit and on the other hand a spectral interpretation of this rate of convergence. Indeed, it can be understood as a modified spectral gap  where the perturbation  from the classical spectral gap is given by the condition number $\kappa(\Lambda)$ which may be interpreted as a measure of deviation from symmetry. In this vein, we  mention the recent work \cite{Patie-Savov} where a similar spectral interpretation of the hypocoercivity phenomena is given for a class of non-self-adjoint Markov semigroups.
\end{rk}


\begin{proof} 
	 We first show the upper bound in \eqref{it:varbd}. Define the synthesis operator $T^* : \ell^2 \to \ell^2(\pi)$ by $\alpha = (\alpha_i) \mapsto T^*(\alpha) = \sum_{i = 0}^{\r} \alpha_i f_i$, where $(f_i)$ are the eigenfunctions of $P$ and $(f_i^*)$ are the unique biorthogonal basis of $(f_i)$ as in Theorem \ref{thm:mcins}. For $1 \leq i \leq \r$, we take $\alpha_i = \lambda_i^n \langle g,f_i^* \rangle_{\pi}$, and denote $(q_i)$ to be the orthonormal eigenfunctions of $Q$, where $f_i = \Lambda q_i$. Note that $||T^*||_{op} \leq ||\Lambda||_{op} < \infty$, since
	$$||T^*(\alpha)|| = \norm{\sum_{i = 0}^{\r} \alpha_i \Lambda q_i} \leq \norm{\Lambda}_{op} \norm{\sum_{i = 0}^{\r} \alpha_i q_i}_{\pi_Q} \leq ||\Lambda||_{op} ||\alpha||_{\ell^2}.$$
	For $g \in \ell^2(\pi)$, we also have
	\begin{align*}
	\sum_{i=0}^{\r} |\langle g,f_i^* \rangle_{\pi}|^2 &= \sum_{i=0}^{\r} |\langle g, (\Lambda^{*})^{-1}q_i \rangle_{\pi}|^2
	= \sum_{i=0}^{\r} |\langle \Lambda^{-1}g, q_i \rangle_{\pi_Q}|^2
	= ||\Lambda^{-1}g||^2_{\pi_Q}
	\leq ||\Lambda^{-1}||^2_{op} ||g||^2_{\pi},
	\end{align*}
	where the third equality follows from Parseval's identity, which leads to
	\begin{align}\label{eq:l2}
	|| P^ng - \pi g||_{\pi}^2 = ||T^*(\alpha)||^2 \leq ||\Lambda||_{op}^2 ||\alpha||_{l^2}^2 \leq ||\Lambda||_{op}^2 ||\Lambda^{-1}||_{op}^2 \lambda_*^{2n}  ||g||^2_{\pi}.
	\end{align}
	Desired upper bound follows from \eqref{eq:l2} and  \begin{equation*}
	\norm{P^n - \pi}_{\ell^2(\pi) \to \ell^2(\pi)} \leq \lambda_*(\widehat{P}P)^{n/2} = \lambda_*(P\widehat{P})^{n/2},
	\end{equation*}
	see e.g.~\cite{LSC97}. The lower bound in \eqref{it:varbd} follows readily from the well-known result that the $n^{th}$ power of the spectral radius $\lambda_*^n$ is less than or equal to the norm of $P^n$ on the reduced space $\ell^2_0(\pi)$. Next, we show that $P$ is non-self-adjoint if and only if $\kappa(\Lambda) > 1$.
We recall from \eqref{eq:interadj} that $PA=A\widehat{P}$ where
$A=\Lambda \widehat{\Lambda}$ is a positive self-adjoint matrix. Thus, by the spectral theorem there exists $T \in \mathrm{GL}_\r$, the general linear group of dimension $\r$, such that $A=T D_A \widehat{T}$ with $D_A$ the diagonal matrix of its eigenvalues where since $\Lambda$ is defined up to a non-zero multiplicative constant we can assume, without loss of generality, that its largest eigenvalue $\lambda_1(A)$=1. Next, we recall from \cite[p.~382]{HJ13} that
\begin{equation}\label{eq:defcond}
  \kappa(\Lambda)=\frac{\sigma_{1}(\Lambda)}{\sigma_{\r}(\Lambda)}=\frac{\sqrt{\lambda_1(A)}}{\sqrt{\lambda_\r(A)}},
\end{equation} where $\sigma_{1}(\Lambda)$ (resp.~$\sigma_{\r}(\Lambda)$) is the largest (resp.~smallest) singular value of $\Lambda$ and $\lambda_\r(A)$ is the smallest eigenvalue which is positive as $\Lambda \in \mathrm{GL}_\r$ and hence  $A \in \mathrm{GL}_\r$.  Thus, $\kappa(\Lambda)=1$ implies that  $D_A=I_{\r}$ where $I_\r$ is the identity matrix, that is $\Lambda \widehat{\Lambda} =A=I_\r$ and, from \eqref{eq:interadj}, we deduce that $P$ is self-adjoint. Conversely, if $P$ is self-adjoint then by means of the same argument used for the proof of Theorem \ref{thm:mcins}\eqref{it:normal}, we have that  $\Lambda$ is unitary and hence $\Lambda \widehat{\Lambda} =I_\r$, that is $\kappa(\Lambda)=1$, which completes the proof of this statement. The claim in \eqref{it:st}  is a straightforward consequence of the Sing-Thompson theorem \cite{Thompson77}.
Next, using \eqref{eq:l2}, we get
\begin{align}\label{eq:vareig}
	\mathrm{Var}_{\pi} \left({\widehat{P}}^n g \right) &\leq \kappa(\widehat{\Lambda})^2 \lambda_*^{2n} \mathrm{Var}_{\pi}(g) = \kappa(\Lambda)^2\lambda_*^{2n} \mathrm{Var}_{\pi}(g), \quad n \in \mathbb{N}_0,
\end{align}
where we used the obvious identity $\kappa(\Lambda) = \kappa(\widehat{\Lambda})$ in the equality. This leads to
\begin{align*}
	|| \delta_x P^n -  \pi ||_{TV}^2 &= \dfrac{1}{4} \E_{\pi}^2 \left| \dfrac{\delta_x P^n}{\pi} - 1 \right|
	\leq \dfrac{1}{4} \mathrm{Var}_{\pi} \left( \dfrac{\delta_x P^n}{\pi} \right)
	= \dfrac{1}{4} \mathrm{Var}_{\pi} \left({\widehat{P}}^n \frac{\delta_x}{\pi} \right)
	\leq \dfrac{1}{4} \kappa(\Lambda)^2 \lambda_{*}^{2n} \mathrm{Var}_{\pi} \left( \frac{\delta_x}{\pi} \right) \\
	&= \dfrac{1}{4} \kappa(\Lambda)^2 \lambda_{*}^{2n} \dfrac{1 - \pi(x)}{\pi(x)} \leq \dfrac{1}{4} \kappa(\Lambda)^2 \lambda_{*}^{2n} \dfrac{1 - \pi_{min}}{\pi_{min}},
\end{align*}
where the first inequality follows from Cauchy-Schwartz
inequality. The proof is completed  by combining the above bound with \eqref{eq:tvfill}.
\end{proof}

\subsection{A numerical example}\label{subsec:numericalex}
To illustrate the previous result, we consider the following example in the $\Mc$ class where the dual transition matrix is given by
\begin{equation}\label{eq:example}
	\widehat{P} = \begin{pmatrix}
	0.275 & 0.7 & 0.005 & 0.02 \\
	0.17 & 0.8 & 0.01 & 0.02 \\
	0 & 0.94 & 0.02 & 0.04 \\
	0 & 0 & 0.95 & 0.05 \\
	\end{pmatrix}.
\end{equation}
Table \ref{tab:example} shows the rate of convergence of  $P$ towards $\pi=(0.18, 0.77,0.03,0.02)$ with $\widehat{P}$ given in \eqref{eq:example}. We observe that for $n = 1,2$, the reversibilization bound $\lambda_{*}(P\widehat{P})^{n/2}$ is smaller while for $n \geq 3$, our upper bound in Theorem \ref{thm:spectralexp} $\kappa_{\Lambda} \lambda_*^n$ is smaller. Also, we point out that since $\max_i p(i,i) = 0.8 > \lambda_* = 0.17$, the Sing-Thompson condition in item \eqref{it:st} of Theorem \ref{thm:spectralexp} holds.

\begin{table}[htbp]
	\centering
	\caption{$\ell^2(\pi)$-rate of convergence of  $P$ }
	\begin{tabular}{rrrr}
		\toprule
		$n$     & $||P^n - \pi||_{\ell^2(\pi) \to \ell^2(\pi)}$ & $\kappa(\Lambda) \lambda_*^n$ & $\lambda_{*}(P\widehat{P})^{n/2}$ \\
		\midrule
		1     & 0.81  & 8.64 & 0.81 \\
		2     & 0.08  & 1.42  & 0.65 \\
		3     & 0.02  & 0.24  & 0.52 \\
		4     & 0.003 & 0.04  & 0.42 \\
		5  	  & 0.0005 & 0.006  & 0.34 \\
		\bottomrule
	\end{tabular}%
		\label{tab:example}%
\end{table}%

\section{Separation cutoff} \label{sec:sc}

As another illustration of the similarity concept, we aim at generalizing, to the non-reversible chains in the class $\mathcal{S}^M$, the separation cutoff criteria established by Diaconis and Saloff-Coste  in \cite{DSC06} and Chen and Saloff-Coste in \cite{CSC15} for reversible birth-death chains. We also offer an alternative necessary and sufficient condition as obtained by Mao et al.~\cite{MZZ16} recently for continuous-time upward skip-free chain with stochastic monotone time-reversal. The term ``cutoff phenomenon" was first formally introduced by Aldous and Diaconis in \cite{AldousDiaconis86}, and total variation cutoff for birth-death chains have been studied by Ding et al.~\cite{DLP10}. To this end, we recall the definition of separation distance of Markov chains, which is used as a standard measure for convergence to equilibrium.
For $n \in \N$, the maximum separation distance $s(n)$ is defined by
$$s(n) = \max_{x,y \in E} \left[ 1 - \dfrac{P^n(x,y)}{\pi(y)}\right] = \max_{x \in E}\, \mathrm{sep}(P^n(x,\cdot),\pi) = \max_{x \in E} s_x(n).$$

Note that separation distance is not a metric. One of its nice feature is its connection to  strong stationary times that we now describe. We say that a strong stationary time $T$ for a Markov chain $X$ with stationary distribution $\pi$ is a randomized stopping time $T$, possibly depending on the initial starting position $x$, if, for all $x,y \in E$,
$$\mP_x(T = n, X_T = y) = \mP_x(T = n) \pi(y).$$
The fastest strong stationary time is a strong stationary time such that for all $n \in \N$, $s_x(n) = \mP(T > n)$.
We now provide a description of the cutoff phenomenon for Markov chains. Recall that the separation mixing times are defined, for any $x \in E$ and $\epsilon>0$, as \[T^s(x,\epsilon) = \min \{ n \geq 0;~\mathrm{sep}(P^n(x,\cdot),\pi) \leq \epsilon\}\]
and	\[T^s(\epsilon) = \min \{ n \geq 0;~s(n) \leq \epsilon\}.\]
A family, indexed by $n \in \N$, of ergodic chains $X^{(n)}$ defined on $E_{\r_n} = \{0,\ldots,\r_n\}$ with transition matrix $P_n$, stationary distribution $\pi_n$ and separation mixing times ${\rm{T}}_n(\epsilon)=T^s_n(\epsilon)$ or $T^s_n(x,\epsilon)$, for some $x\in E$, is said to present a separation cutoff if there is a positive sequence  $(t_n)$ such that for all $\epsilon \in (0,1)$,
$$\lim_{n \to \infty} \dfrac{\rm{T}_n(\epsilon)}{t_n} = 1.$$
The family has a $(t_n, b_n)$ separation cutoff if the sequences $(t_n)$ and $(b_n)$ are positive, $b_n/t_n \to 0$ and
for all $\epsilon \in (0,1)$,
$$\limsup_{n \to \infty} \dfrac{|{\rm{T}}_n(\epsilon) - t_n|}{b_n} < \infty.$$
Let us now write $\mathrm{R}_n=(I - P_n)^{-1}_{\ell^2_0}$  for the centered resolvent, that is, the resolvent operator restricted to $l^2_0(\pi)$. The main result of this section is the following.

\begin{theorem}\label{thm:sepcutoffsf}
	Suppose that, for each  $n \geq 1$,   $X^{(n)} \in \S_{\r_n}^{M}$  and let $(\theta_{n,i})_{i=1}^{\r_n}$ be the non-zero eigenvalues of $I - P_n$. Define
	\[ 
\underline{\theta}_n = \min_{1 \leq i \leq \r_n} \theta_{n,i}, \quad \rho_n^2 = \sum_{i=1}^{\r_n} \dfrac{1-\theta_{n,i}}{\theta_{n,i}^2}.\]
	Then the family of chains $(X^{(n)})$ with transition kernel $(P_n)$, all started from $0$, has a separation cutoff if and only if $\mathrm{Tr}(\mathrm{R}_n) \underline{\theta}_n \to \infty$ if and only if $T_n^s(0,\epsilon) \underline{\theta}_n \to \infty$. In this case there is a $(\mathrm{Tr}(\mathrm{R}_n), \max\{\rho_n,1\})$ separation cutoff.
\end{theorem}

We begin the proof of Theorem \ref{thm:sepcutoffsf} by giving an important lemma that gives a lower bound on the mixing time in terms of the eigenvalues of $I - P$. The corresponding result for reversible and ergodic Markov chain can be found in \cite[Theorem $12.4$]{LPW09}.

\begin{lemma}\label{lem:lowerbdmix}
	For an ergodic chain $X \in \S$ and $\epsilon \in (0,1)$, denote by $(\theta_i)_{i = 0}^{\r}$  the eigenvalues of $I - P$ arranged in ascending order and $\underline{\theta} = \min_{i \neq 0} \theta_i$. We have
	$$T^s(\epsilon) \geq \left(\underline{\theta}^{-1}-1 \right) \log \left( \dfrac{1}{2\epsilon}\right).$$
\end{lemma}
\begin{proof}
	Take $f$ to be an eigenfunction of $P$ associated to an arbitrary  eigenvalue $\lambda \neq 1$, where $\lambda \in (\lambda_i)_{i=0}^{\r}$ which is the set of (real) eigenvalues of $P$ arranged in descending order. Note that $f$ is orthogonal to $\1$, since
	$$\langle Pf,\1 \rangle_{\pi} = \lambda \langle f,\1 \rangle_{\pi} = \langle f,\widehat{P}\1 \rangle_{\pi} = \langle f,\1 \rangle_{\pi},$$
	where the last equality follows from stochasticity of $\widehat{P}$. By writing
	$||f||_{\infty} = \max_{x \in E} |f(x)| = f(x^*)$, we have $$|\lambda^n f(x)| = |P^n f(x)| = \left| \sum_{y \in E} P^n(x,y)f(y) - \pi(y)f(y) \right| \leq 2 ||f||_{\infty}  ||P^n(x,\cdot) - \pi||_{\mathrm{TV}} \leq 2 ||f||_{\infty} s(n),$$
	where the first inequality follows from the definition of total variation distance $||P^n(x,\cdot) - \pi||_{\mathrm{TV}}$, and the second inequality comes from the result that total variation distance is less than or equal to separation distance, see e.g.~\cite[Lemma $6.13$]{LPW09}. Taking $n = T^s(\epsilon)$ and $x = x^*$, the above yields
	$|\lambda|^{T^s(\epsilon)} \leq 2 \epsilon$, which leads to
	$$ \frac{1-|\lambda|}{|\lambda|} T^s(\epsilon)\geq - \log(|\lambda|) T^s(\epsilon)  \geq - \log( 2\epsilon ).$$
	The proof is completed by specializing to $|\lambda| = \max\{|\lambda_1|,|\lambda_{\r}|\}$.
\end{proof}

The next lemma gives a necessary condition for separation cutoff in terms of the spectral information.

\begin{lemma}[Necessary condition for separation cutoff]\label{lem:sepcutoffnec}
	Suppose that, for each  $n \geq 1$,  $X^{(n)} \in \Sf_{\r_n}$  and let $(\theta_{n,i})_{i=1}^{\r_n}$ be the eigenvalues of $I - P_n$. For the $n^{th}$ chain in the family, define $\underline{\theta}_n = \min_{1 \leq i \leq \r_n} \theta_{n,i}$ and $T_n^s(\epsilon)$ to be the separation mixing time. If the family of chain with transition kernel $(P_n)$ exhibits a separation cutoff, then $T_n^s(\epsilon) \underline{\theta}_n \to \infty$.
\end{lemma}

\begin{proof}
Let $T_n^s = T_n^s(0.25)$, that is, choosing  $\epsilon = 0.25$. If $T_n^s(\epsilon) \underline{\theta}_n$ is bounded above in $n$ by $c>0$, then by Lemma \ref{lem:lowerbdmix} we have
\[ \dfrac{T_n^s(\epsilon)}{T_n^s} \geq \dfrac{\underline{\theta}_n^{-1}-1 }{T_n^s}\log \left( \dfrac{1}{2\epsilon}\right) \geq c \log \left( \dfrac{1}{2\epsilon}\right),\]
so $\dfrac{T_n^s(\epsilon)}{T_n^s} \to \infty$ as $\epsilon \to 0$, which implies that there is no separation cutoff.
\end{proof}

We are now ready to complete the proof of Theorem \ref{thm:sepcutoffsf}. Denote by $P_n^k$ the distribution of the $n^{th}$ chain at time $k$, and by $\pi_n$  the stationary measure of the $n^{th}$ chain. We also write $t_n = \mathrm{Tr}(\mathrm{R}_n)$. It is known, see e.g.~\cite[Theorem $1.4$]{Fill}, that
	$$\mathrm{sep}(P_n^k,\pi_n) = \mP(T_n > k),$$
	where $T_n$ is the fastest strong stationary time of the $n^{th}$ chain, which is equal in distribution to a $\r_n$-fold convolution of geometric random variables each with success probability $\theta_{n,i}$ for $i = 1,\ldots,\r_n$, mean $t_n$ and variance $\rho_n^2$. The key to establish the proof is the following.
	\begin{align}\label{eq:eigcompare}
	\rho_n^2 = \underline{\theta}_n^{-2} \sum_{i=1}^{\r_n} \dfrac{\left(1-\theta_{n,i}\right) \underline{\theta}_n^2}{\theta_{n,i}^2} \leq \underline{\theta}_n^{-2}\sum_{i=1}^{\r_n} \dfrac{\underline{\theta}_n}{\theta_{n,i}} = \underline{\theta}_n^{-1} t_n,
	\end{align}
	where we use the facts that $\theta_{n,i} \geq 0$ and $\underline{\theta}_n/\theta_{n,i} \leq 1$ in the inequality.
		Assume that $t_n \underline{\theta}_n \to \infty$, which together with \eqref{eq:eigcompare} yields $\rho_n/t_n \to 0$. The rest of the argument are similar to the ones developed in  the proof of \cite[Theorem $5.1$]{DSC06}. For sake of completeness we now provide its main ingredients. First, by means of Chebyshev's inequality, we have
	\begin{equation}\label{eq:cheby}
		t_n - (\epsilon^{-1} - 1)^{1/2} \rho_n \leq T^s_n(0,\epsilon) \leq t_n + (\epsilon^{-1} - 1)^{1/2} \rho_n.
	\end{equation}
	This shows that the family of chain exhibits a separation cutoff if we divide \eqref{eq:cheby} by $t_n$ and take $n \to \infty$. On the other hand, separation cutoff implies $T_n^s(0,\epsilon) \underline{\theta}_n \to \infty$ by Lemma \ref{lem:sepcutoffnec}, so it remains to show $T_n^s(0,\epsilon) \underline{\theta}_n \to \infty$ implies $t_n \underline{\theta}_n \to \infty$. Using $t_n \underline{\theta}_n \geq 1$, \eqref{eq:eigcompare} yields $\rho_n \leq t_n$, and together with the upper bound of \eqref{eq:cheby} leads to
	$$T^s_n(0,\epsilon) \leq t_n + (\epsilon^{-1} - 1)^{1/2} \rho_n \leq t_n (1 + (\epsilon^{-1} - 1)^{1/2}),$$
	so $t_n \underline{\theta}_n \to \infty$ holds if and only if $T_n^s(0,\epsilon) \underline{\theta}_n \to \infty$. It follows from \cite[Remark $1.1$]{CSC15} that there is a $(t_n, \max\{\rho_n,1\})$ separation cutoff. Precisely, \eqref{eq:cheby} gives
	$$|T_n^s(0,\epsilon) - t_n| \leq (\epsilon^{-1} - 1)^{1/2} \rho_n + 1,$$
	and a $(t_n, \max\{\rho_n,1\})$ cutoff is observed by noting that $\theta_{n,i} \leq 2$, $t_n \geq n/2$ and $\rho_n/t_n \to 0 $.


\subsection{$\ell^p$-cutoff}

We proceed by investigating the $\ell^p$-cutoff for fixed $p \in (1,\infty)$ for the class $\S$. Recall that 
 Chen and Saloff-Coste \cite[Theorem $4.2,4.3$]{CSC07} have shown that for a family of \textit{normal} ergodic transition kernel $P_n$, the max-$\ell^p$ cutoff is equivalent to the {\emph{spectral gap times mixing time}} going to infinity. We can extend their result to the case of the non-normal chains in $ \S $ as follows.
\begin{theorem}[Max-$\ell^p$ cutoff]
	Suppose that, for each  $n \geq 1$,   $X^{(n)} \in \mathcal{S}_{\r_n}$ with compact transition kernel $P_n \stackrel{\Lambda_n}{\sim} Q_n$ and stationary measure $\pi_n$, and let $\lambda_{n,*}$ be the second largest eigenvalue in modulus of $P_n$. Assume that
	$$\sup_{n \geq 1} ||\Lambda_n||_{op} \: ||\Lambda^{-1}_n||_{op} < \infty.$$
	Fix $p \in (1,\infty)$ and $\epsilon > 0$. Consider the max-$\ell^p$ distance to stationarity $$f_n(t) = \sup_{x \in E_{\r_n}} \norm{\dfrac{p^t_n(x,\cdot)}{\pi_n} - 1}_{\ell^p(\pi)} $$ and define
	$$t_n = \inf\{t > 0;\: f_n(t) \leq \epsilon\}, \quad \theta_{n,*} = - \log \lambda_{n,*} \quad \textrm{and } \mathcal{F} = \{f_n;\: n=1,2,\ldots\}.$$
	Assume that each $n$, $f_n(t) \to 0$ as $t \to \infty$ and $t_n \to \infty$. Then the family $\mathcal{F}$ has a max-$\ell^p$ cutoff if and only if $t_n \theta_{n,*} \to \infty$. In this case there is a $(t_n, \max\{1,\theta_{n,*}^{-1}\})$ cutoff.
\end{theorem}
The proof in \cite[Theorem $4.2,4.3$]{CSC07} works nicely as long as we have Lemma \ref{lem:l2lp} below, which gives a two-sided control on the $\ell^p(\pi)$ norm of $P^n - \pi$. 
The following lemma is then the key to the proof.
\begin{lemma}\label{lem:l2lp}
	Suppose that $X \in  \S $ on $E$ with transition kernel $P \stackrel{\Lambda}{\sim} Q$. Fix $p \in (1,\infty)$. Then, for  any $n \in \mathbb{N}$, we have
	\begin{align}
		2^{-1+\theta_p} \lambda_*(P)^{n\theta_p} &\leq \norm{P^n - \pi}_{\ell^p(\pi) \to \ell^p(\pi)} \leq 2^{|1-2/p|} (\kappa(\Lambda) \lambda_*(P)^n)^{1-|1-2/p|}, \label{eq:lp}
	\end{align}
	where $\theta_p \in [1/2,1]$ and $\kappa(\Lambda) = \norm{\Lambda}_{op} \: ||\Lambda^{-1}||_{op}$.
\end{lemma}

\begin{proof}
	By the  Riesz-Thorin interpolation theorem, see e.g.~\cite[equation $3.4$]{CSC07}, we have
	$$\norm{P^n - \pi}_{\ell^p(\pi) \to \ell^p(\pi)} \leq 2^{|1-2/p|} \norm{P^n - \pi}_{\ell^2(\pi) \to \ell^2(\pi)}^{1-|1-2/p|},$$
	which when combined with Theorem \ref{thm:spectralexp} gives the upper bound of \eqref{eq:lp}. Next, to show the lower bound in \eqref{eq:lp}, we use another version of the Riesz-Thorin interpolation theorem,  see e.g.~\cite[Lemma $4.1$]{CSC07}, to  get
	$$\norm{P^n - \pi}_{\ell^p(\pi) \to \ell^p(\pi)} \geq2^{-1+\theta_p} \norm{P^n - \pi}_{\ell^2(\pi) \to \ell^2(\pi)}^{\theta_p} \geq 2^{-1+\theta_p} \lambda_*(P)^{n\theta_p},$$
	where we use Theorem \ref{thm:spectralexp} in the second inequality. This completes the proof.
\end{proof}

\bibliographystyle{plain}
\bibliography{skipfree2}

\begin{thebibliography}{10}

\bibitem{AW89}
J.~Abate and W.~Whitt.
\newblock Spectral {T}heory for {S}kip-{F}ree {M}arkov {C}hains.
\newblock {\em Probability in the Engineering and Informational Sciences},
  3(1):77¡V88, 1989.

\bibitem{Adikahri86}
A.~Adikahri.
\newblock {\em Skip free processes}.
\newblock PhD thesis, University of California, Berkeley, 1986.

\bibitem{AldousDiaconis86}
D.~Aldous and P.~Diaconis.
\newblock Shuffling cards and stopping times.
\newblock {\em The American Mathematical Monthly}, 93(5):333--348, 1986.

\bibitem{Anderson91}
W.J. Anderson.
\newblock {\em Continuous-time Markov Chains: An Applications-oriented
  Approach}.
\newblock Applied probability. Springer-Verlag, 1991.

\bibitem{ASM03}
S.~Asmussen.
\newblock {\em Applied probability and queues}, volume~51 of {\em Applications
  of Mathematics (New York)}.
\newblock Springer-Verlag, New York, second edition, 2003.
\newblock Stochastic Modelling and Applied Probability.

\bibitem{APW18}
F.~Avram, P.~Patie, and J.~Wang.
\newblock Purely {E}xcessive {F}unctions and {H}itting {T}imes of
  {C}ontinuous-{T}ime {B}ranching {P}rocesses.
\newblock {\em Methodology and Computing in Applied Probability}, Jan 2018.

\bibitem{BGR82}
P.~J. Brockwell, J.~Gani, and S.~I. Resnick.
\newblock Birth, {I}mmigration and {C}atastrophe {P}rocesses.
\newblock {\em Advances in Applied Probability}, 14(4):709--731, 1982.

\bibitem{CSC07}
G.~Y. Chen and L.~Saloff-Coste.
\newblock The cutoff phenomenon for ergodic {M}arkov processes.
\newblock {\em Electron. J. Probab.}, 13:no. 3, 26--78, 2008.

\bibitem{CSC15}
G.~Y. Chen and L.~Saloff-Coste.
\newblock Computing cutoff times of birth and death chains.
\newblock {\em Electron. J. Probab.}, 20:no. 76, 47, 2015.

\bibitem{ChoiThesis}
M.~C.H. Choi.
\newblock {\em Analysis of non-reversible {M}arkov chains}.
\newblock PhD thesis, Cornell University, 2017.

\bibitem{CW05}
K.~L. Chung and J.~B. Walsh.
\newblock {\em Markov processes, {B}rownian motion, and time symmetry}, volume
  249 of {\em Grundlehren der Mathematischen Wissenschaften [Fundamental
  Principles of Mathematical Sciences]}.
\newblock Springer, New York, second edition, 2005.

\bibitem{Con90}
J.~B. Conway.
\newblock {\em A course in functional analysis}, volume~96 of {\em Graduate
  Texts in Mathematics}.
\newblock Springer-Verlag, New York, second edition, 1990.

\bibitem{FD}
P.~Diaconis and J.~A. Fill.
\newblock Strong stationary times via a new form of duality.
\newblock {\em Ann. Probab.}, 18(4):1483--1522, 1990.

\bibitem{DSC06}
P.~Diaconis and L.~Saloff-Coste.
\newblock Separation cut-offs for birth and death chains.
\newblock {\em Ann. Appl. Probab.}, 16(4):2098--2122, 2006.

\bibitem{DLP10}
J.~Ding, E.~Lubetzky, and Y.~Peres.
\newblock Total variation cutoff in birth-and-death chains.
\newblock {\em Probability theory and related fields}, 146(1-2):61, 2010.

\bibitem{DM76}
H.~Dym and H.~P. McKean.
\newblock {\em Gaussian processes, function theory, and the inverse spectral
  problem}.
\newblock Academic Press [Harcourt Brace Jovanovich, Publishers], New
  York-London, 1976.
\newblock Probability and Mathematical Statistics, Vol. 31.

\bibitem{dynkin}
E.~B. Dynkin.
\newblock The boundary theory of {M}arkov processes (discrete case).
\newblock {\em Uspehi Mat. Nauk}, 24(2 (146)):3--42, 1969.

\bibitem{Fe54}
W.~Feller.
\newblock Diffusion processes in one dimension.
\newblock {\em Trans. Amer. Math. Soc.}, 77:1--31, 1954.

\bibitem{F66}
W.~Feller.
\newblock Infinitely divisible distributions and {B}essel functions associated
  with random walks.
\newblock {\em SIAM J. Appl. Math.}, 14:864--875, 1966.

\bibitem{Fe71}
W.~Feller.
\newblock {\em An Introduction to Probability Theory and Its Applications}.
\newblock Number v. 2 in An Introduction to Probability Theory and Its
  Applications. Wiley, 1971.

\bibitem{Fill91}
J.~A. Fill.
\newblock Eigenvalue bounds on convergence to stationarity for nonreversible
  {M}arkov chains, with an application to the exclusion process.
\newblock {\em Ann. Appl. Probab.}, 1(1):62--87, 1991.

\bibitem{Fill}
J.~A. Fill.
\newblock On hitting times and fastest strong stationary times for skip-free
  and more general chains.
\newblock {\em J. Theoret. Probab.}, 22(3):587--600, 2009.

\bibitem{Gorbachuk1997}
M.~L. Gorbachuk and V.~I. Gorbachuk.
\newblock {\em M. {G}. {K}rein's lectures on entire operators}, volume~97 of
  {\em Operator Theory: Advances and Applications}.
\newblock Birkh\"auser Verlag, Basel, 1997.

\bibitem{GP80}
Priscilla Greenwood and Jim Pitman.
\newblock Fluctuation identities for {L}\'{e}vy {P}rocesses and {S}plitting at
  the {M}aximum.
\newblock {\em Advances in Applied Probability}, 12(4):893--902, 1980.

\bibitem{HJ13}
R.~A. Horn and C.~R. Johnson.
\newblock {\em Matrix analysis}.
\newblock Cambridge University Press, Cambridge, second edition, 2013.

\bibitem{Karlin-McGregor_Spectral}
S.~Karlin and J.~L. McGregor.
\newblock The differential equations of birth-and-death processes, and the
  {S}tieltjes moment problem.
\newblock {\em Trans. Amer. Math. Soc.}, 85:489--546, 1957.

\bibitem{Ke71}
J.~Keilson.
\newblock Log-concavity and log-convexity in passage time densities of
  diffusion and birth-death processes.
\newblock {\em J. Appl. Probability}, 8:391--398, 1971.

\bibitem{KSK}
J.~G. Kemeny, J.~L. Snell, and A.~W. Knapp.
\newblock {\em Denumerable {M}arkov chains}.
\newblock Springer-Verlag, New York-Heidelberg-Berlin, second edition, 1976.
\newblock With a chapter on Markov random fields, by David Griffeath, Graduate
  Texts in Mathematics, No. 40.

\bibitem{KL83}
J.~T. Kent and N.~T. Longford.
\newblock An eigenvalue decomposition for first hitting times in random walks.
\newblock {\em Z. Wahrsch. Verw. Gebiete}, 63(1):71--84, 1983.

\bibitem{LR54}
W.~Ledermann and G.~E.~H. Reuter.
\newblock Spectral theory for the differential equations of simple birth and
  death processes.
\newblock {\em Philos. Trans. Roy. Soc. London. Ser. A.}, 246:321--369, 1954.

\bibitem{Lefevre2013}
C.~Lef\`evre, P.~Patie, and S.~Vakeroudis.
\newblock Fluctuation theory of one-sided {Ornstein-Uhlenbeck} type processes.
\newblock {\em Working paper}, 2016.

\bibitem{LPW09}
D.~A. Levin, Y.~Peres, and E.~L. Wilmer.
\newblock {\em Markov chains and mixing times}.
\newblock American Mathematical Society, Providence, RI, 2009.
\newblock With a chapter by James G. Propp and David B. Wilson.

\bibitem{MZZ16}
Y.~H. Mao, C.~Zhang, and Y.~H. Zhang.
\newblock Separation cutoff for upward skip-free chains.
\newblock {\em J. Appl. Probab.}, 53(1):299--306, 03 2016.

\bibitem{Miclo}
L.~Miclo.
\newblock On absorption times and {D}irichlet eigenvalues.
\newblock {\em ESAIM Probab. Stat.}, 14:117--150, 2010.

\bibitem{Miclo15}
L.~Miclo.
\newblock On the markovian similarity.
\newblock Preprint, March 2016.

\bibitem{Patie-Savov}
P.~Patie and M.~Savov.
\newblock
  \href{https://www.researchgate.net/publication/277712416_Spectral_expansions_of_non-self-adjoint_generalized_Laguerre_semigroups}{Spectral
  expansion of non-self-adjoint generalized {Laguerre semigroups}}.
\newblock {\em Submitted}, page 162p. (current version), 2015.

\bibitem{Patie-Wev}
P.~Patie and C.~Van~Weverberg.
\newblock Exit times of continuous state branching processes with immigration.
\newblock {\em Working paper}, 2016.

\bibitem{Patie-Vigon}
P.~Patie and V.~Vigon.
\newblock On completely assymetric markov processes.
\newblock {\em Working paper}, 2016.

\bibitem{Patie-Wang}
P.~Patie and J.~Wang.
\newblock Potential and fluctuation theory of galton-watson process with
  immigration.
\newblock {\em working paper}, 2018.

\bibitem{RS80}
M.~Reed and B.~Simon.
\newblock {\em Methods of modern mathematical physics. {I}}.
\newblock Academic Press, Inc. [Harcourt Brace Jovanovich, Publishers], New
  York, second edition, 1980.
\newblock Functional analysis.

\bibitem{R59}
G.~E.~H. Reuter.
\newblock Denumerable {M}arkov processes. {II}.
\newblock {\em J. London Math. Soc.}, 34:81--91, 1959.

\bibitem{RR}
G.~O. Roberts and J.~S. Rosenthal.
\newblock Geometric ergodicity and hybrid {M}arkov chains.
\newblock {\em Electron. Comm. Probab.}, 2:no.\ 2, 13--25 (electronic), 1997.

\bibitem{LSC97}
L.~Saloff-Coste.
\newblock Lectures on finite {M}arkov chains.
\newblock In {\em Lectures on probability theory and statistics
  ({S}aint-{F}lour, 1996)}, volume 1665 of {\em Lecture Notes in Math.}, pages
  301--413. Springer, Berlin, 1997.

\bibitem{Sieg76}
D.~Siegmund.
\newblock The equivalence of absorbing and reflecting barrier problems for
  stochastically monotone {M}arkov processes.
\newblock {\em Ann. Probability}, 4(6):914--924, 1976.

\bibitem{Spitzer}
F.~Spitzer.
\newblock A combinatorial lemma and its application to probability theory.
\newblock {\em Trans. Amer. Math. Soc.}, 82:323--339, 1956.

\bibitem{SVH04}
F.~W. Steutel and K.~van Harn.
\newblock {\em Infinite divisibility of probability distributions on the real
  line}, volume 259 of {\em Monographs and Textbooks in Pure and Applied
  Mathematics}.
\newblock Marcel Dekker, Inc., New York, 2004.

\bibitem{Thompson77}
R.~C. Thompson.
\newblock Singular values, diagonal elements, and convexity.
\newblock {\em SIAM J. Appl. Math.}, 32(1):39--63, 1977.

\bibitem{Thorisson00}
H.~Thorisson.
\newblock {\em Coupling, Stationarity, and Regeneration}.
\newblock Probability and Its Applications. Springer New York, 2000.

\bibitem{V09}
C.~Villani.
\newblock Hypocoercivity.
\newblock {\em Mem. Amer. Math. Soc.}, 202(950):iv+141, 2009.

\bibitem{Viskov00}
O.~V. Viskov.
\newblock A random walk with an upper-continuous component, and the {L}agrange
  inversion formula.
\newblock {\em Teor. Veroyatnost. i Primenen.}, 45(1):166--175, 2000.

\bibitem{Woess}
W.~Woess.
\newblock {\em Denumerable {M}arkov chains}.
\newblock EMS Textbooks in Mathematics. European Mathematical Society (EMS),
  Z\"urich, 2009.

\bibitem{Y01}
R.~M. Young.
\newblock {\em An introduction to nonharmonic {F}ourier series}.
\newblock Academic Press, Inc., San Diego, CA, first edition, 2001.

\end{thebibliography}

\end{document}